\documentclass[noams]{compositio}
\usepackage{color}
\usepackage{amsmath,amssymb,amsfonts,graphicx,fancyhdr}
\usepackage{array}
\usepackage{bbm}
 \usepackage[colorlinks=true,citecolor=blue,linkcolor=red]{hyperref}
\usepackage{indentfirst}
\usepackage{float}

\newtheorem{thm}{Theorem}[section]
\newtheorem{remark}[thm]{Remark}
\newtheorem{example}[thm]{Example}

\newtheorem{notation}[thm]{Notation}

\newtheorem{defn}[thm]{Definition}
\newtheorem{lem}[thm]{Lemma}
\newtheorem{prop}[thm]{Proposition}
\newtheorem{cor}[thm]{Corollary}

\newcommand{\D}{\ensuremath{{\rm d}}}
\begin{document}

\title{Rank $n$ swapping algebra for the $\operatorname{PSL}(n, \mathbb{R})$ Hitchin component}
\author{Zhe Sun}
\email{sunzhe1985@gmail.com}
\address{Room 106, Jing Zhai Building, Tsinghua University, Hai Dian District, Beijing, 100084, China}

\thanks{The research leading to these results has received funding from the European Research Council under the {\em European Community}'s seventh Framework Programme
(FP7/2007-2013)/ERC {\em  grant agreement}.}

\begin{abstract}
F. Labourie [arXiv:1212.5015] characterized the Hitchin components for $\operatorname{PSL}(n, \mathbb{R})$ for any $n>1$ by using the swapping algebra, where the swapping algebra should be understood as a ring equipped with a Poisson bracket. We introduce the rank $n$ swapping algebra, which is the quotient of the swapping algebra by the $(n+1)\times(n+1)$ determinant relations. The main results are the well-definedness of the rank $n$ swapping algebra and the ``cross-ratio'' in its fraction algebra. As a consequence, we use the sub fraction algebra of the rank $n$ swapping algebra generated by these "cross-ratios" to characterize the $\operatorname{PSL}(n, \mathbb{R})$ Hitchin component for a fixed $n>1$.
We also show the relation between the rank $2$ swapping algebra and the cluster $\mathcal{X}_{\operatorname{PGL}(2,\mathbb{R}),D_k}$-space.
\end{abstract}

\maketitle

\section{Introduction}
\subsection{Background}
Let $S$ be a connected oriented closed surface of genus $g>1$. When $G$ is a reductive Lie group, the character variety is
$$R(S,G):= \{\;homomorphisms \; \rho: \; \pi_1(S) \rightarrow G\;\} \verb|//| G,$$
where the group $G$ acts on homomorphisms above by conjugation, and the
quotient is taken in the sense of geometric invariant theory \cite{MFK94}. When $G=\operatorname{PSL}(2, \mathbb{R})$, the character variety $R(S,\operatorname{PSL}(2, \mathbb{R}))$ has $4g-3$ connected components \cite{G88}. Two of these components correspond to all discrete faithful homomorphisms from $\pi_1(S)$ to $\operatorname{PSL}(2, \mathbb{R})$. By the uniformization theorem, any one of the two components is diffeomorphic to the Teichm\"uller space of complex structures on $S$ up to isotopy.
 For $n \geq 2$, we define {\em $n$-Fuchsian representation} to be a representation $\rho$, which can be written as $\rho = i \circ \rho_0$, where $\rho_0$ is a discrete faithful representation of $\pi_1(S)$ with values in $\operatorname{PSL}(2, \mathbb{R})$ and $i$ is the irreducible representation of $\operatorname{PSL}(2, \mathbb{R})$ in $\operatorname{PSL}(n, \mathbb{R})$.
In \cite{H92}, N. Hitchin found one of the connected components of the character variety $R(S,\operatorname{PSL}(n, \mathbb{R}))$, which contains the $n$-Fuchsian representations, called {\em Hitchin component} and denoted by $H_n(S)$. By N. Hitchin \cite{H92}, the GIT quotient of the Hitchin component $H_n(S)$ coincides with its usual topological quotient. Furthermore, the Hitchin component $H_n(S)$ is diffeomorphic to a ball $\mathbb{R}^{(2g-2)(n^2-1)}$.

A decade later, F. Labourie \cite{L06} and O. Guichard \cite{Gu08} showed that
every $\rho$ in the Hitchin component $H_n(S)$ is one to one associated to a
$\rho$-equivariant ($\xi_\rho(\gamma x) = \rho(\gamma)\xi_\rho(x)$) hyperconvex Frenet curve $\xi_\rho$ from the boundary at infinity of $\pi_1(S)$---$\partial_\infty \pi_1(S)$ to $\mathbb{RP}^{n-1}$, where {\em hyperconvex} means that for any pairwise distinct points $(x_1, ... , x_p)$ with $p \leq n$,
the sum $\xi(x_1) + ... + \xi(x_p)$ is direct. Let $\xi^*_\rho$ be its associated $\rho$-equivariant osculating hyperplane curve from $\partial_\infty \pi_1(S)$ to $\mathbb{RP}^{{n-1}*}$.
Let $\tilde{\xi_{\rho}}$ ($\tilde{\xi^*_{\rho}}$ resp.) be the lifts of $\xi_{\rho}$ ($\xi^*_{\rho}$ resp.) with values in $\mathbb{R}^n$ ($\mathbb{R}^{n*}$ resp.).
F. Labourie defined the {\em weak cross ratio} $\mathbb{B}_\rho$ of four different points $x,y,z,t$ in $\partial_\infty \pi_1(S)$:
\begin{equation}
	\mathbb{B}_{\rho}(x,y,z,t) = \frac{\left<\left.\tilde{\xi}(x)\right\vert \tilde{\xi^*}(z)\right> }{\left<\left.\tilde{\xi}(x)\right\vert \tilde{\xi^*}(t)\right>} \cdot \frac{\left<\left.\tilde{\xi}(y)\right\vert \tilde{\xi^*}(t)\right>}{\left<\left.\tilde{\xi}(y)\right\vert \tilde{\xi^*}(z)\right>}.
\end{equation}
Such cross ratios are the only cross ratios, called the rank $n$ cross ratios, that satisfy some symmetry properties, normalisation properties, multiplicative cocycle identities, $\pi_1(S)$-invariant properties and $\mathbb{R}^n$-linear algebraic properties \cite{L07}. Therefore, the space of the rank $n$ cross ratios identifies with the Hitchin component $H_n(S)$.

Later on, F. Labourie \cite{L12} defined the swapping algebra to characterize the union of the Hitchin components $\bigcup_{n=2}^\infty H_n(S)$.
The swapping algebra is defined on the ordered pair of points of a subset $\mathcal{P} \subseteq S^1$. More precisely, we represent an ordered pair $(x, y)$ of $\mathcal{P}$ by the expression $xy$, and we consider the ring $\mathcal{Z}(\mathcal{P}) := \mathbb{K}[\{xy\}_{\forall x,y \in \mathcal{P}}]/\left(\{xx| \forall x \in \mathcal{P}\}\right)$ over a field $\mathbb{K}$ of characteristic zero. Then we equip $\mathcal{Z}(\mathcal{P})$ with a Poisson bracket $\{\cdot,\cdot\}$, called the swapping bracket,
by extending the formula on generators for any $rx, sy \in \mathcal{P}$:
\begin{equation}
\{rx, sy\} = \mathcal{J}(rx, sy) \cdot ry \cdot sx ,
\end{equation}
to $\mathcal{Z}(\mathcal{P})$ by using Leibniz's rule.
We will define the {\em linking number} $\mathcal{J}(rx, sy)$ in Section \ref{sa}.
Therefore, the {\em swapping algebra} of $\mathcal{P}$ is $(\mathcal{Z}(\mathcal{P}), \{\cdot , \cdot\})$.
Let $x,y,z,t$ belong to $\mathcal{P}$ so that $x \neq t$ and $y \neq z$. The {\em cross fraction} determined by $(x,y,z,t)$ is the element:
$$
[x,y,z,t] :=\frac{xz}{xt} \cdot \frac{yt}{yz}.
$$
Let $\mathcal{B}(\mathcal{P})$ be the sub fraction ring of $\mathcal{Z}(\mathcal{P})$ generated by all the cross fractions.
Then, the {\em swapping multifraction algebra of $\mathcal{P}$} is $(\mathcal{B}(\mathcal{P}), \{\cdot , \cdot\})$. Let $\mathcal{R}$ be the subset of $\partial_\infty \pi_1(S)$
given by the end points of periodic geodesics.
F. Labourie consider a natural homomorphism $I$ from $\mathcal{B}(\mathcal{R})$ to $C^\infty(H_n(S))$ by extending the following formula on generators to $\mathcal{B}(\mathcal{R})$:
\begin{equation}
I([x,y,z,t]) = \mathbb{B}_{\rho}(x,y,z,t).
\end{equation}

\begin{thm}\label{lswapabg}
{\sc[F. Labourie \cite{L12}]} Let $S$ be a connected oriented closed surface of genus $g>1$. Let $\{\cdot, \cdot\}$ be the swapping bracket. For $n\geq 2$, let $\{\cdot, \cdot\}_{S}$ be the Atiyah-Bott-Goldman Poisson bracket \cite{AB83}\cite{G84}  of the Hitchin component $H_n(S)$. If $\Gamma_1,...,\Gamma_k,...$ is a vanishing sequence of finite index subgroups of $\pi_1(S)$. Let  $S_k= \mathbb{H}^2/\Gamma_k$, vanishing means that any two primitive representatives of $\pi_1(S)$ in the sequence $S_1,...,S_k,...$ intersect simply at zero or one point at last. For any $b_0,b_1 \in \mathcal{B}(\mathcal{R})$, we have $$\lim_{k\rightarrow \infty}\{I(b_0),I(b_1)\}_{S_k} = I\circ\{b_0,b_1\}.$$
\end{thm}
The above theorem is true for any integer $n>1$, therefore the swapping multifraction algebra $(\mathcal{B}(\mathcal{R}),\{\cdot,\cdot\})$ asymptotically characterizes the union of Hitchin components $\bigcup_{n=2}^\infty H_n(S)$.

F. Labourie also showed that, for the space $\mathcal{L}_n$ of the Drinfeld-Sokolov reduction \cite{DS85}\cite{Se91} on the space of $\operatorname{PSL}(n,\mathbb{R})$-Hitchin opers with trivial holonomy, the natural homomorphism $i$ from the swapping multifraction algebra $\mathcal{B}(S^1)$ to the function space $C^\infty(\mathcal{L}_n)$ is Poisson with respect to the swapping bracket and the Poisson bracket corresponding to second Gelfand-Dickey symplectic structure.

Both the homomorphism $I$ and the homomorphism $i$ have large kernels arising from linear algebra of $\mathbb{R}^n$. Is the swapping algebra $(\mathcal{Z}(\mathcal{P}), \{\cdot , \cdot\})$ still well-defined after divided by these corresponding linear algebraic relations? Is the associated sub fraction algebra generated by all the cross fractions well-defined? These two questions are the main focus of this paper.

\subsection{Rank $n$ swapping algebra and the main results}
For $n\geq 2$, let $R_n(\mathcal{P})$ be the ideal of $\mathcal{Z}(\mathcal{P})$ generated by  $$\left\{ D \in \mathcal{Z}(\mathcal{P}) \; |  \; D = \det \left(\begin{array}{lcr}
                                        x_1 y_1 & ... & x_1 y_{n+1} \\
                                        ... & ... & ... \\
                                        x_{n+1} y_1 & ... & x_{n+1} y_{n+1}
                                      \end{array} \right) ,  \forall x_1, ..., x_{n+1}, y_1,..., y_{n+1} \in \mathcal{P} \right\}.$$
Let $\mathcal{Z}_n(\mathcal{P})$ be the quotient ring $\mathcal{Z}(\mathcal{P})/R_n(\mathcal{P})$. The following two theorems are the main results of this paper, which will be proven in Section \ref{rsa} and \ref{sectionint}.
By induction on corresponding positions of the points on the circle, we prove the following theorem.
\begin{thm}
\label{wd}
For $n \geq 2$, $R_n(\mathcal{P})$ is a Poisson ideal with respect to the swapping bracket, thus $\mathcal{Z}_{n}(\mathcal{P})$ inherits a Poisson bracket from the swapping bracket.
\end{thm}
It then follows the Theorem \ref{wd} that the {\em rank $n$ swapping algebra of $\mathcal{P}$} is the compatible pair $(\mathcal{Z}_{n}(\mathcal{P}), \{\cdot,\cdot\})$.
For the well-definedness of the cross fractions of the ring $\mathcal{Z}_n(\mathcal{P})$, by using very classical geometric invariant theory \cite{CP76}\cite{W39} and Lie group cohomology \cite{CE48}, we prove the following theorem.
\begin{thm}
\label{int}
For $n \geq 2$ , the quotient ring $\mathcal{Z}_n(\mathcal{P})$  is an integral domain.
\end{thm}
Let $\mathcal{B}_n(\mathcal{P})$ be the sub fraction ring of $\mathcal{Z}_n(\mathcal{P})$ generated by all the cross fractions.
Then, the {\em rank $n$ swapping multifraction algebra of $\mathcal{P}$} is the pair $(\mathcal{B}_n(\mathcal{P}), \{\cdot , \cdot\})$.
Thus, the homomorphism $I$ naturally factors through $\mathcal{B}_n(\mathcal{P})$ because of the rank $n$ cross ratio conditions \cite{L07}, which provides a homomorphism $$I_n :\mathcal{B}_n(\mathcal{R}) \rightarrow C^\infty(H_n(S)).$$
Then we can replace $I$ by $I_n$ in Theorem \ref{lswapabg}.
Therefore, for a fixed $n\geq2$, the rank $n$ swapping multifraction algebra $(\mathcal{B}_n(\mathcal{P}), \{\cdot,\cdot\})$ is the Poisson algebra which characterizes $H_n(S)$.
But the homomorphism $I_n$ is not injective, since the image of the cross fractions are $\pi_1(S)$ invariant. Still, the non-injectivity and the asymptotic behavior of $I_n$ are two obstructions to characterize $(H_n(S),\omega_{ABG})$ exactly. We suggest that these two obstructions are worth of being investigated.

For the homomorphism $i$, we do not have these two obstructions.
Similarly, we also have a homomorphism $$i_n :\mathcal{B}_n(S^1) \rightarrow C^\infty(\mathcal{L}_n)$$
induced from the homomorphism $i$. The homomorphism $i_n$ is Poisson by Theorem 10.7.2 in \cite{L12} and Theorem \ref{wd}, and injective by Theorem \ref{inv1}. As a conseqence, the rank $n$ swapping multifraction algebra $(\mathcal{B}_n(S^1), \{\cdot,\cdot\})$ should be regarded as the dual of $\mathcal{W}_n$ algebra.
\subsection{Rank $2$ swapping algebra and the cluster $\mathcal{X}_{\operatorname{PGL}(2,\mathbb{R}),D_k}$-space}
Let $S$ be a connected oriented surface with non-empty boundary and a finite set $P$ of special points on boundary, considered modulo isotopy. The rank $n$ swapping algebra also relates to the Fock-Goncharov's cluster-$\mathcal{X}_{\operatorname{PGL}(n,\mathbb{R}),S}$ space. V. Fock and A. Goncharov \cite{FG06} introduced the positive structure in sense of \cite{L94}\cite{L98} and the cluster algebraic structure for  the moduli space $\mathcal{X}_{\operatorname{PGL}(n,\mathbb{R}),S}$ of framed local systems of the surface $S$. The positive part of the moduli space $\mathcal{X}_{\operatorname{PGL}(n,\mathbb{R}),S}$ is related to the Hitchin component $H_n(S)$. (For the surface $S$ with boundary or punctures, we can still define $H_n(S)$, but the monodromy around a boundary component is conjugated to an upper or lower triangular totally positive matrix.)
Moreover, they introduced a special coordinate system for the cluster $\mathcal{X}_{\operatorname{PGL}(n,\mathbb{R}),S}$-space in \cite{FG06} Section 9, which generalizes Thurston's shear coordinates for Teichm\"uller space \cite{T86}. (The Fock-Goncharov coordinates are also used in case of the closed surface $S$ of genus $g>1$ \cite{FD14}.) This coordinate system is local, because it depends on the ideal triangulation $\mathcal{T}$. Moreover, the coordinate system for $\mathcal{T}$ gives us a split tori $\mathbb{T}_{\mathcal{T}}$ of $\mathcal{X}_{\operatorname{PGL}(n,\mathbb{R}),S}$. The space $\mathcal{X}_{\operatorname{PGL}(n,\mathbb{R}),S}$ is a variety glued by all these $\mathbb{T}_{\mathcal{T}}$, and the transition function from $\mathbb{T}_{\mathcal{T}}$ to another $\mathbb{T}_{\mathcal{T}'}$ is defined by a positive rational transformation corresponding to a composition of flips, where each flip is a composition of mutations in its cluster algebraic structure. The positive structure of $\mathcal{X}_{\operatorname{PGL}(n,\mathbb{R}),S}$ arises from the positivity of the rational transformations.

Let $D_k$ be a disc with $k$ special points on the boundary. In the last section, we will prove the following theorem.
\begin{thm}
Given an ideal triangulation $\mathcal{T}$ of $D_k$, there is an injective and Poisson homomorphims from the fraction algebra generated by the Fock-Goncharov coordinates for the cluster $\mathcal{X}_{\operatorname{PGL}(2,\mathbb{R}),D_k}$-space to the rank $2$ swapping multifraction algebra $(\mathcal{B}_n(\mathcal{P}), \{\cdot,\cdot\})$, with respect to the natural Fock-Goncharov Poisson bracket and the swapping bracket.
\end{thm}
Then we will show that the cluster dynamic of the cluster $\mathcal{X}_{\operatorname{PGL}(2,\mathbb{R}),D_k}$-space can also be interpreted by the rank $2$ swapping algebra. As a consequence, the natural Fock-Goncharov Poisson bracket does not depend on the triangulations.
The above theorem is generalized for  $\mathcal{X}_{\operatorname{PGL}(n,\mathbb{R}),D_k}$-space in the following papers. For $n=3$, in Chapter 3 of \cite{Su14}, we showed a complicated homomorphism, where $k$ flags of $\mathbb{RP}^2$ correspond to the set $\mathcal{P}$ with $k$ elements. For a general $n$, the homomorphism is discussed in \cite{Su15}, where the set $\mathcal{P}$ has $(n-1)\cdot k$ elements, each flag of $\mathbb{RP}^{n-1}$ corresponding to $n-1$ points near each other on the boundary $S^1$.

Therefore, the rank $n$ swapping algebra provides the links among the Hitchin component $H_n(S)$, $\mathcal{W}_n$ algebra and the cluster $\mathcal{X}_{\operatorname{PGL}(2,\mathbb{R}),D_k}$-space.
\subsection{Further discussions}
In the upcoming paper \cite{Su1511}, we will define a quantized version of the rank $n$ swapping algebra. The quantization of $\mathcal{X}_{D_k, \operatorname{PSL}(n,\mathbb{R})}$ by Fock-Goncharov \cite{FG06} \cite{FG09} is embedded into our quantization of the rank $n$ swapping algebra. We will glue the rank $n$ swapping algebras to characterize the cluster $\mathcal{X}_{S, \operatorname{PSL}(n,\mathbb{R})}$-space for the surface $S$ in general. We expect to build a TQFT and some geometric invariants from the rank $n$ swapping algebra.

In \cite{Su1412}, we relate the rank $n$ swapping algebra to the discrete integrable system of the configuration space of N-twisted polygon in $\mathbb{RP}^{n-1}$ \cite{FV93}\cite{SOT10}\cite{KS13}. When $n=2$, there is a bi-hamiltonian structure for the configuration space of N-twisted polygon in $\mathbb{RP}^{n-1}$. This was conjectured in \cite{SOT10} for $n=3$. We expect that there exsits a bi-hamiltonian structure for $n$ in general.

\section{Swapping algebra revisited}
\label{sa}
In this section, we will recall some basic definitions about the swapping algebra introduced by F. Labourie in Section 2 of \cite{L12}. The new part of this section is that we take care of the compatibilities of the rings related to $\mathcal{Z}(\mathcal{P})$ and the swapping bracket, particularly the sub fraction ring $\mathcal{B}(\mathcal{P})$ generated by ``cross ratios''.
 \subsection{Linking number}
\begin{defn}{\sc[linking number]}
Let $(r, x, s, y)$ be a quadruple of four points in $S^1$. The {\em linking number} between $rx$ and $sy$ is
\begin{equation}
\mathcal{J}(rx, sy) = \frac{1}{2} \cdot \left( \sigma(r-x) \cdot \sigma(r-y) \cdot \sigma(y-x)
- \sigma(r-x) \cdot \sigma(r-s ) \cdot \sigma(s-x)\right),
\end{equation}
such that for any $a \in \mathbb{R}$, we define $\sigma(a)$ as follows. Remove any point $o$ different from $r,x,s,y$ in $S^1$ in order to get an interval $]0,1[$. Then the points $r,x,s,y \in S^1$ correspond to the real numbers in $]0,1[$, $\sigma(a)= -1; 0; 1$ whenever $a < 0$; $a = 0$; $a > 0$ respectively.
\end{defn}
\begin{figure}\centering
\includegraphics[width=4in]{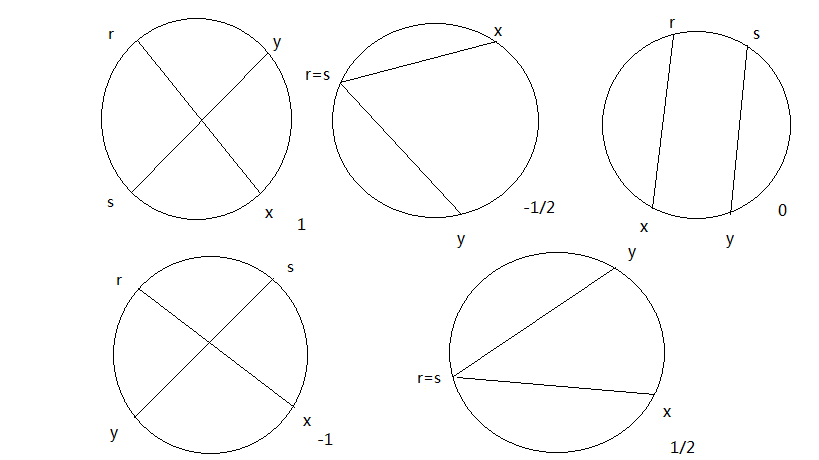}
\caption{\label{swapswap1} Linking number between $rx$ and $sy$.}
\end{figure}
In fact, the value of $\mathcal{J}(rx, sy)$ belongs to $\{0, \pm1, \pm \frac{1}{2}\}$, depends on the corresponding positions of $r,x,s,y$ and does not depend on the choice of the point $o$. In Figure \ref{swapswap1}, we describe five possible values of $\mathcal{J}(rx, sy)$.

\subsection{Swapping algebra}
Let $\mathcal{P}$ be a cyclic subset of $S^1$, we represent an ordered pair $(r, x)$ of $\mathcal{P}$ by the expression $rx$. Then we consider the associative commutative ring $$\mathcal{Z}(\mathcal{P}) := \mathbb{K}[\{xy\}_{\forall x,y \in \mathcal{P}}]/\{xx| \forall x \in \mathcal{P}\}$$ over a field $\mathbb{K}$ of characteristic $0$, where $\{xy\}_{\forall x,y \in \mathcal{P}}$ are variables.
Then we equip $\mathcal{Z}(\mathcal{P})$ with a swapping bracket.

\begin{defn}{\sc[swapping bracket \cite{L12}]}
\label{defn2.2}
{\em The swapping bracket} over $\mathcal{Z}(\mathcal{P})$ is defined by extending the following formula for any $rx$, $sy$ in $\mathcal{P}$ to $\mathcal{Z}(\mathcal{P})$ by using {\em Leibniz's rule}:
\begin{equation}
\{rx, sy\} = \mathcal{J}(rx, sy) \cdot ry \cdot sx.
\end{equation}
\end{defn}
By direct computations, F. Labourie proved the following theorem.
\begin{thm}
{\sc[F. Labourie \cite{L12}]} The swapping bracket is Poisson.
\end{thm}

\begin{defn}{\sc[swapping algebra]}
{\em The swapping algebra of $\mathcal{P}$} is the ring $\mathcal{Z}(\mathcal{P})$ equipped with the swapping bracket, denoted by $(\mathcal{Z}(\mathcal{P}),\{\cdot, \cdot\})$.
\end{defn}

\subsection{Swapping multifraction algebra}
In this subsection, we consider the rings related to $\mathcal{Z}(\mathcal{P})$ and their compatibilities with the swapping bracket.
\begin{defn}{\sc[closed under swapping bracket]}
For a ring $R$, if $\forall a,b \in R$, we have $\{a,b\} \in R$, then we say that $R$ is {\em closed under swapping bracket}.
\end{defn}

 Since $\mathcal{Z}(\mathcal{P})$ is an integral domain, let $\mathcal{Q}(\mathcal{P})$ be the total fraction ring of $\mathcal{Z}(\mathcal{P})$.
  By Leibniz's rule, we have $\{a, \frac{1}{b}\} = -\frac{\{a, b\}}{b^2} $, thus the swapping bracket is well defined on $\mathcal{Q}(\mathcal{P})$. Therefore we have
the following definition.
\begin{defn}{\sc[swapping fraction algebra of $\mathcal{P}$]}
The {\em swapping fraction algebra of $\mathcal{P}$} is the ring $\mathcal{Q}(\mathcal{P})$ equipped with the induced swapping bracket, denoted by $(\mathcal{Q}(\mathcal{P}),\{\cdot, \cdot\})$.
\end{defn}
\begin{defn}{\sc[Cross fraction]}
Let $x,y,z,t$ belong to $\mathcal{P}$ so that $x \neq t$ and $y \neq z$. The {\em cross fraction} determined by $(x,y,z,t)$ is the element of $\mathcal{Q}(\mathcal{P})$:
\begin{equation}
[x,y,z,t] :=\frac{xz}{xt} \cdot \frac{yt}{yz}.
\end{equation}
\end{defn}

\begin{remark}
\label{cfraction}
Notice that the cross fractions verify the following {\em cross-ratio conditions} \cite{L07}:

Symmetry: $[a,b,c,d] = [b,a,d,c]$,

Normalisation: $[a,b,c,d] = 0$ if and only if $a=c$ or $b=d$,

Normalisation: $[a,b,c,d] = 1$ if and only if $a=b$ or $c=d$,

Cocycle identity: $[a,b,c,d] \cdot [a,b,d,e] = [a,b,c,e]$,

Cocycle identity: $[a,b,d,e] \cdot [b,c,d,e] = [a,c,e,f]$.

\end{remark}

Let $\mathcal{CR}(\mathcal{P}) = \{[x,y,z,t] \in \mathcal{Q}(\mathcal{P}) \; |\; \forall x,y,z,t \in \mathcal{P}, \; x \neq t, y \neq z\}$ be the set of all the cross-fractions in $\mathcal{Q}(\mathcal{P})$.
Let $\mathcal{B}(\mathcal{P})$ be the subring of $\mathcal{Q}(\mathcal{P})$ generated by $\mathcal{CR}(\mathcal{P})$.
\begin{prop}\label{thm2.11}
The ring $\mathcal{B}(\mathcal{P})$ is closed under swapping bracket.
\end{prop}
\begin{proof}
By Leibniz's rule, $\forall c_1,..., c_n, d_1,...,d_m \in \mathcal{Z}(\mathcal{P})$
\begin{equation}
\frac{\{c_1\cdot \cdot \cdot c_n, d_1\cdot \cdot \cdot d_m\}}{c_1\cdot \cdot \cdot c_n\cdot d_1\cdot \cdot \cdot d_m} = \sum_{i=1}^{n} \sum_{j=1}^{m} \frac{\{c_i, d_j\}}{c_i\cdot  d_j},
\end{equation}
we only need to show that for any two elements $[x,y,z,t]$ and $[u,v,w,s]$ in $\mathcal{CR}(\mathcal{P})$,
where $x\neq t$, $y\neq z$, $u\neq s$, $v\neq w$ in $\mathcal{P}$, then $\{\frac{xz}{xt}\cdot \frac{yt}{yz} , \frac{uw}{us} \cdot \frac{vs}{vw}\} \in \mathcal{B}(\mathcal{P})$. Let $e_1 = xz$, $e_2 = \frac{1}{xt}$, $e_3 = yt$, $e_4 = \frac{1}{yz}$, $h_1 = uw$, $h_2 = \frac{1}{us}$, $h_3 = vs$, $h_4 = \frac{1}{vw}$. By the definition of the swapping bracket, we have $$\frac{\{e_1, h_1\}}{e_1 \cdot h_1} = \mathcal{J}(xz,uw)\cdot \frac{xw}{xz} \cdot \frac{uz}{uw} \in \mathcal{B}(\mathcal{P}).$$ Then by the Leibniz's rule, we deduce that for any $e,h \in \mathcal{Z}(\mathcal{P})$, we have
$$\frac{\{e, \frac{1}{h}\}}{e/h} = -\frac{\{e, h\}}{e\cdot h},\; \frac{\{\frac{1}{e}, h\}}{h/e} = -\frac{\{e, h\}}{e\cdot h},\; \frac{\{\frac{1}{e}, \frac{1}{h}\}}{1/(eh)} = \frac{\{e, h\}}{e\cdot h}. $$
So for any $i,j = 1,2,3,4$, we have $\frac{\{e_i, h_j\}}{e_i \cdot h_j} \in \mathcal{B}(\mathcal{P})$. Since $e_1 e_2 e_3 e_4$ and $h_1 h_2 h_3 h_4$ are also in $\mathcal{B}(\mathcal{P})$, so
$$\{e_1 e_2 e_3 e_4, h_1 h_2 h_3 h_4\}  = \sum_{i=1}^{4}\sum_{j=1}^{4} \frac{\{e_i, h_j\}}{e_i \cdot h_j}\cdot (e_1 e_2 e_3 e_4 h_1 h_2 h_3 h_4) \in \mathcal{B}(\mathcal{P}).$$
Finally, we conclude that $\mathcal{B}(\mathcal{P})$ is closed under swapping bracket.
\end{proof}
\begin{defn}{\sc[swapping multifraction algebra of $\mathcal{P}$]}
 {\em The swapping multifraction algebra of $\mathcal{P}$} is the ring $\mathcal{B}(\mathcal{P})$ equipped with the swapping bracket, denoted by $(\mathcal{B}(\mathcal{P}),\{\cdot, \cdot\})$.
\end{defn}

\section{Rank $n$ swapping algebra}
\label{rsa}
Swapping algebra $\left(\mathcal{Z}(\mathcal{P}),\{\cdot, \cdot\}\right)$ corresponds to $\bigcup_{n=2}^{\infty}H_n(S)$. In this section, we define the rank $n$ swapping algebra $\mathcal{Z}_n(\mathcal{P})$, in order to restrict the correspondence for a fixed $n$. In theorem \ref{rwdefined}, we will prove that the ring $\mathcal{Z}_n(\mathcal{P})$ is compatible with the swapping bracket.
\subsection{The rank $n$ swapping ring $\mathcal{Z}_n(\mathcal{P})$}
\label{section 2.3}
\begin{notation}
Let $$\Delta((x_1, ..., x_{n+1}), (y_1,..., y_{n+1})) = \det \left(\begin{array}{lcr}
                                        x_1 y_1 & ... & x_1 y_{n+1} \\
                                        ... & ... & ... \\
                                        x_{n+1} y_1 & ... & x_{n+1} y_{n+1}
                                      \end{array} \right).$$
\end{notation}
Inspired by linear algebra for $\mathbb{R}^n$, and the space of the rank $n$ cross-ratios identified with the Hitchin component $H_n(S)$ \cite{L07}, we define the rank $n$ swapping ring as follows.

\begin{defn}{\sc[The rank $n$ swapping ring $\mathcal{Z}_n(\mathcal{P})$]}
For $n\geq 2$, let $R_n(\mathcal{P})$ be the ideal of $\mathcal{Z}(\mathcal{P})$ generated by  $\{ D \in \mathcal{Z}(\mathcal{P}) \; \vert  \; D = \Delta((x_1, ..., x_{n+1}), (y_1,..., y_{n+1})) ,  \forall x_1, ..., x_{n+1}, y_1,..., y_{n+1} \in \mathcal{P} \,\}$.

The {\em rank $n$ swapping ring} $\mathcal{Z}_n(\mathcal{P})$ is the quotient ring $\mathcal{Z}(\mathcal{P})/R_n(\mathcal{P})$.
\end{defn}
\begin{remark}
Decomposing the determinant $D$ in the first row, we have by induction that
\begin{equation}
R_2(\mathcal{P})\supseteq R_3(\mathcal{P}) \supseteq... \supseteq R_n(\mathcal{P}).
\end{equation}
\end{remark}

\subsection{Swapping bracket over $\mathcal{Z}_n(\mathcal{P})$}
We will prove by induction the first fundamental theorem of the rank $n$ swapping algebra.
\begin{thm}
\label{rwdefined}{\sc[\textbf{First main result}]}
For $n \geq 2$, the ideal $R_n(\mathcal{P})$ is a Poisson ideal with respect to the swapping bracket. Thus the ring $\mathcal{Z}_{n}(\mathcal{P})$ inherits a Poisson bracket from the swapping bracket.
\end{thm}
\begin{proof}
The above theorem is equivalent to say that for any $h \in R_{n}(\mathcal{P})$ and any $f \in \mathcal{Z}(\mathcal{P})$, we have $\{f, h\} \in R_{n}(\mathcal{P})$ where $n \geq 2$.
By the Leibniz's rule of the swapping bracket, it suffices to prove the case where $f = ab \in \mathcal{Z}(\mathcal{P})$, $h = \Delta((x_1, ..., x_{n+1}), (y_1,..., y_{n+1}))$. The points $x_1,...,x_{n+1}$ $(y_1,...,y_{n+1} \; resp. )$ should be different from each other in $\mathcal{P}$, otherwise $h = 0$. Therefore the theorem follows from Lemma \ref{swcal}.
\end{proof}
\begin{figure}\centering
\includegraphics[width=4in]{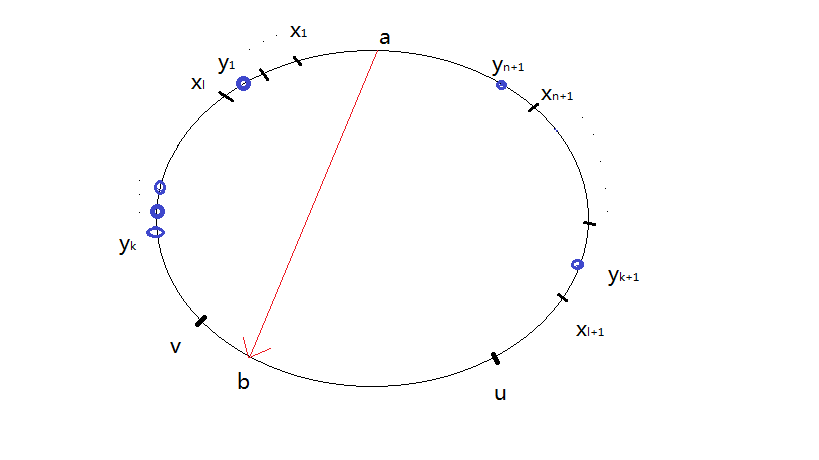}
\caption{\label{swapabdet1} $\{ab, \Delta((x_1, ..., x_{n+1}), (y_1,..., y_{n+1}))\}$.}
\end{figure}
\begin{lem}
\label{swcal}
Let $n \geq 2$. Let $x_1,...,x_{n+1}$ $(y_1,...,y_{n+1} \; resp. )$ in $\mathcal{P}$ be different from each other and ordered anticlockwise, $a,b$ belong to $\mathcal{P}$ and $x_{1},...,x_{l}, y_{1},..., y_{k}$ are on the right side of $\overrightarrow{ab}$ (include coinciding with $a$ or $b$) as illustrated in Figure \ref{swapabdet1}.
  Let $u$ ($v$ resp.) be strictly on the left (right resp.) side of $\overrightarrow{ab}$. Let
 \begin{equation}
   \begin{aligned}
&\Delta^R(a,b)= \sum_{d=1}^l \mathcal{J}(ab, x_d u)\cdot x_d b \cdot \Delta((x_1, ...,x_{d-1}, a, x_{d+1},...,x_{n+1}), (y_1,..., y_{n+1}))
  \\&+ \sum_{d=1}^k \mathcal{J}(ab, u y_d)\cdot a y_d \cdot \Delta((x_1, ...,x_{n+1}), (y_1,...,y_{d-1}, b, y_{d+1},..., y_{n+1})),
  \end{aligned}
 \end{equation}
We obtain that
 \begin{equation}
   \begin{aligned}
   \label{equabdet1}
   &\{ab, \Delta((x_1, ..., x_{n+1}), (y_1,..., y_{n+1}))\}
   = \Delta^R(a,b).
   \end{aligned}
   \end{equation}
\end{lem}

\begin{proof}
The main idea of the proof is to consider the change of $\{ab, \Delta((x_1, ..., x_{n+1}), (y_1,..., y_{n+1}))\}$ when $ab$ moves topologically in the circle with special points $a,b,x_1,...,x_{n+1},y_1,...,y_{n+1}$.

We will prove that
$$\{ab, \Delta((x_1, ..., x_{n+1}), (y_1,..., y_{n+1}))\} = \Delta^R(a,b)$$
   by induction on the number of elements of $\{x_1,...,x_{n+1},y_1,...,y_{n+1}\}$ on the right side of $\overrightarrow{ab}$
 (includes coinciding with $a$ or $b$), which is $m=l+k$. Let $S_{n+1}$ be the permutation group of $\{1,...,n+1\}$, the signature of $\sigma \in S_{n+1}$ denoted by $sgn(\sigma)$, is defined as $1$ if $\sigma$ is even and $-1$ if $\sigma$ is odd. Then we have
\begin{equation}
\Delta((x_1, ..., x_{n+1}), (y_1,..., y_{n+1})) = \sum_{\sigma \in S_{n+1}} sgn(\sigma) \prod_{i=1}^{n+1}  x_{i} y_{\sigma(i)}.
\end{equation}
By the Leibniz's rule, we obtain that
\begin{equation}
\begin{aligned}
\label{equabdet}
&\{ab, \Delta((x_1, ..., x_{n+1}), (y_1,..., y_{n+1})) \} = \sum_{\sigma \in S_{n+1}} sgn(\sigma) \prod_{i=1}^{n+1}  x_{i} y_{\sigma(i)} \sum_{j=1}^{n+1} \frac{\{ab, x_j y_{\sigma(j)}\}}{x_j y_{\sigma(j)}}
\\&= \sum_{\sigma \in S_{n+1}} sgn(\sigma) \prod_{i=1}^{n+1}  x_{i} y_{\sigma(i)} \left( \sum_{j=1}^{n+1} \frac{\mathcal{J}(ab, x_j y_{\sigma(j)})\cdot a y_{\sigma(j)} \cdot x_j b}{x_j y_{\sigma(j)}}\right).
\end{aligned}
\end{equation}

\begin{figure}\centering
\includegraphics[width=4in]{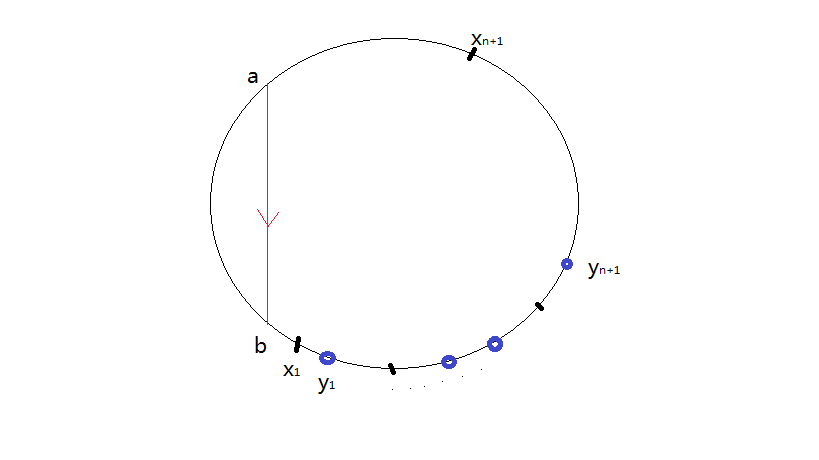}
\caption{\label{swapabdetaxi} $m=0$.}
\end{figure}

When $m=0$ as illustrated in Figure \ref{swapabdetaxi}, since $\mathcal{J}(ab, x_j y_{\sigma(j)})=0$, we have $\{ab, x_j y_{\sigma(j)}\}=0$ for any $j=1,...,n+1$ and any $\sigma \in S_{n+1}$. By Equation \ref{equabdet}, we have
$$\{ab, \Delta((x_1, ..., x_{n+1}), (y_1,..., y_{n+1}))\} =0= \Delta^R(a,b)$$
    in this case.

Suppose $$\{ab, \Delta((x_1, ..., x_{n+1}), (y_1,..., y_{n+1}))\}
   = \Delta^R(a,b)$$ for $m=q \geq0$.

When $m=q+1$, suppose that $x_l$ is the first point of $\{x_1,...,x_{n+1}, y_1,...,y_{n+1}\}$ on the right side of $\overrightarrow{ab}$ (include coinciding with $a$ or $b$) with respect to the clockwise orientation.

\begin{enumerate}
  \item
\begin{figure}\centering
\includegraphics[width=4in]{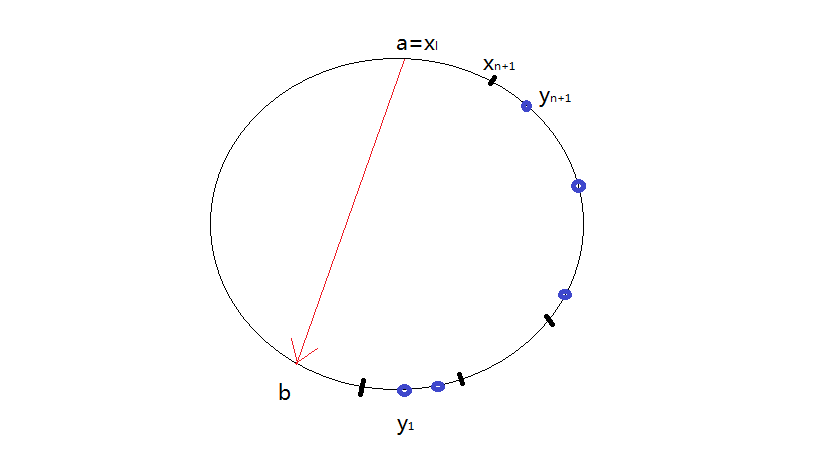}
\caption{\label{Ta} $m=q+1$ case $a= x_l$.}
\end{figure}
If $x_l$ coincides with $a$ as illustrated in Figure \ref{Ta}, then $m=1$. So we have $\mathcal{J}(ab, x_l y_{\sigma(l)})=\frac{1}{2}$ and $\mathcal{J}(ab, x_j y_{\sigma(j)})=0$ for $j \neq l$. By Equation \ref{equabdet}, we have
$$\{ab, \Delta((x_1, ..., x_{n+1}), (y_1,..., y_{n+1}))\}
   =\frac{1}{2} \cdot ab \cdot \Delta((x_1, ..., x_{n+1}), (y_1,..., y_{n+1}))= \Delta^R(a,b).$$

  \item
If $x_l$ does not coincide with $a$, we move $b$ clockwise to the point $b'$, such that $b' \neq x_l$ and the intersection between $\{x_1,...,x_{n+1}, y_1,...,y_{n+1}\}$ and the arc $bb'$ is $x_l$ as illustrated in Figure \ref{T13}.

Then $\{ab', \Delta((x_1, ..., x_{n+1}), (y_1,..., y_{n+1}))\}$ corresponds to the case $m=q$. Thus we have
\begin{equation}
\begin{aligned}
\label{equabdet17}
&\{ab', \Delta((x_1, ..., x_{n+1}), (y_1,..., y_{n+1}))\}
\\&= \sum_{d=1}^{l-1} \mathcal{J}(ab', x_d u)\cdot x_d b \cdot \Delta((x_1, ...,x_{d-1}, a, x_{d+1},...,x_{n+1}), (y_1,..., y_{n+1}))
  \\&+ \sum_{d=1}^k \mathcal{J}(ab', u y_d)\cdot a y_d \cdot \Delta((x_1, ...,x_{n+1}), (y_1,...,y_{d-1}, b, y_{d+1},..., y_{n+1})).
  \end{aligned}
  \end{equation}

\begin{figure}\centering
\includegraphics[width=4in]{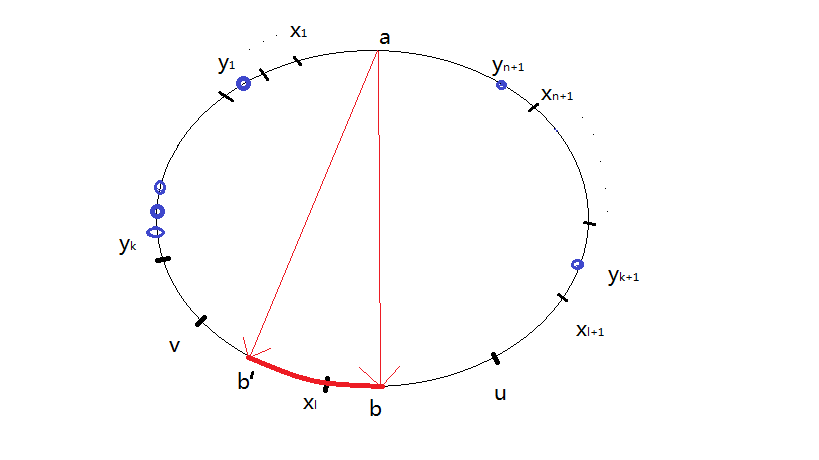}
\caption{\label{T13} $m=q+1$ case $a \neq x_l$.}
\end{figure}

On the other hand, by Equation \ref{equabdet},
\begin{equation}
\begin{aligned}
&\{ab', \Delta((x_1, ..., x_{n+1}), (y_1,..., y_{n+1}))\}
\\&= \sum_{\sigma \in S_{n+1}} sgn(\sigma) \prod_{i=1}^{n+1}  x_{i} y_{\sigma(i)} \left( \sum_{j=1}^{n+1} \frac{\mathcal{J}(ab', x_j y_{\sigma(j)})\cdot a y_{\sigma(j)} \cdot x_j b'}{x_j y_{\sigma(j)}}\right)
\end{aligned}
\end{equation}
 is a polynomial of $ab', x_1 b',...,x_{n+1} b'$, denoted by $P(ab', x_1 b',...,x_{n+1} b')$. Then
 $$P(ab, x_1 b,...,x_{n+1} b) = \sum_{\sigma \in S_{n+1}} sgn(\sigma) \prod_{i=1}^{n+1}  x_{i} y_{\sigma(i)} \left(\sum_{j=1}^{n+1} \frac{\mathcal{J}(ab', x_j y_{\sigma(j)})\cdot a y_{\sigma(j)} \cdot x_j b}{x_j y_{\sigma(j)}}\right).$$
 By the cocycle identity \cite{L12}:  $\mathcal{J}(ab,xy) - \mathcal{J}(ab',xy) = \mathcal{J}(b'b,xy)$, we have
\begin{equation}
\begin{aligned}
\label{equabdet19}
&\{ab, \Delta((x_1, ..., x_{n+1}), (y_1,..., y_{n+1}))\} - P(ab, x_1 b,...,x_{n+1} b)
\\&= \sum_{\sigma \in S_{n+1}} sgn(\sigma) \prod_{i=1}^{n+1}  x_{i} y_{\sigma(i)} \left(\sum_{j=1}^{n+1} \frac{\left(\mathcal{J}(ab, x_j y_{\sigma(j)}) - \mathcal{J}(ab', x_j y_{\sigma(j)})\right)\cdot a y_{\sigma(j)} \cdot x_j b}{x_j y_{\sigma(j)}}\right)
\\& = \sum_{\sigma \in S_{n+1}} sgn(\sigma) \prod_{i=1}^{n+1}  x_{i} y_{\sigma(i)} \left(\sum_{j=1}^{n+1} \frac{\mathcal{J}(b'b, x_j y_{\sigma(j)}) \cdot a y_{\sigma(j)} \cdot x_j b}{x_j y_{\sigma(j)}}\right)
\end{aligned}
\end{equation}
Since $\mathcal{J}(b'b, x_j y_{\sigma(j)}) = 0$ when $j \neq l$, $\mathcal{J}(b'b, x_l y_{\sigma(l)}) = \mathcal{J}(ab, x_l u)$.
So the above sum equals to
\begin{equation}
\begin{aligned}
\label{equabdet20}
&\sum_{\sigma \in S_{n+1}} sgn(\sigma) \prod_{i=1}^{n+1}  x_{i} y_{\sigma(i)} \cdot \frac{\mathcal{J}(b'b, x_l y_{\sigma(l)}) \cdot a y_{\sigma(l)} \cdot x_l b}{x_l y_{\sigma(l)}}
\\& = \mathcal{J}(ab, x_l u)  \cdot x_l b \cdot \Delta((x_1, ...,x_{l-1}, a, x_{l+1},..., x_{n+1}), (y_1,..., y_{n+1})).
\end{aligned}
\end{equation}
Since $\mathcal{J}(ab', x_d u) = \mathcal{J}(ab, x_d u)$ for $d =1,...,l-1$, $\mathcal{J}(ab', u y_d) = \mathcal{J}(ab, u y_d)$ for $d = 1,...,k$, by Equations \ref{equabdet17} \ref{equabdet19} \ref{equabdet20},
we have $$\{ab, \Delta((x_1, ..., x_{n+1}), (y_1,..., y_{n+1}))\} = \Delta^R(a,b)
$$
in this case.
\end{enumerate}

When $y_k$ is the first point of $\{x_1,...,x_{n+1}, y_1,...,y_{n+1}\}$ on the right side of $\overrightarrow{ab}$ (include coinciding with $a$ or $b$) with respect to clockwise orientation, the result follows the similar argument.
By induction, we have $$\{ab, \Delta((x_1, ..., x_{n+1}), (y_1,..., y_{n+1}))\} = \Delta^R(a,b)
$$
in general.

  Finally, we conclude that $$\{ab, \Delta((x_1, ..., x_{n+1}), (y_1,..., y_{n+1}))\} = \Delta^R(a,b).$$
\end{proof}
\begin{remark}
We can also consider $\Delta^L(a,b)$ similar to $\Delta^R(a,b)$ with respect to the left side of $\overrightarrow{ab}$,
The equation $\Delta^R(a,b) =  \Delta^L(a,b)$ follows
\begin{equation}
  \begin{aligned}
  \label{equabxy}
  &\sum_{i=1}^{n+1} a y_i \cdot \Delta((x_1, ..., x_{n+1}), (y_1,..., y_{i-1}, b, y_{i+1},..., y_{n+1}))
  \\&= \sum_{i=1}^{n+1} x_i b \cdot \Delta((x_1,..., x_{i-1}, a, x_{i+1},..., x_{n+1}), (y_1, ..., y_{n+1})).
  \end{aligned}
  \end{equation}
\end{remark}

\begin{example}
\begin{figure}\centering
\includegraphics[width=4in]{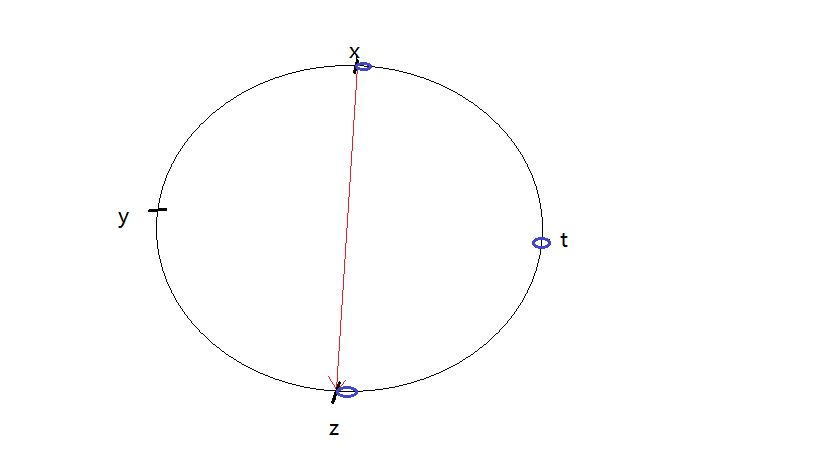}
\caption{\label{swapabdet3} Example.}
\end{figure}
As shown in Figure \ref{swapabdet3}, we have
\begin{equation}
\{xz, \Delta((x, z, y), (z,x,t))\} = - xt \cdot \Delta((x, z, y), (z,x,z))  = 0.
\end{equation}
\end{example}
Therefore, the swapping bracket over $\mathcal{Z}_{n}(\mathcal{P})$ is well defined for $n \geq 2$.

\begin{defn}{\sc[rank $n$ swapping algebra of $\mathcal{P}$]}
For $n \geq 2$, the {\em rank $n$ swapping algebra of $\mathcal{P}$} is the rank $n$ swapping ring $\mathcal{Z}_n(\mathcal{P})$ equipped with the swapping bracket, denoted by $(\mathcal{Z}_n(\mathcal{P}), \{\cdot, \cdot\})$.
\end{defn}

\section{The ring $\mathcal{Z}_n(\mathcal{P})$ is an integral domain}
\label{sectionint}
In this section, we will show that the cross fractions are well-defined in the fraction ring of $\mathcal{Z}_n(\mathcal{P})$, by proving that the ring $\mathcal{Z}_n(\mathcal{P})$ is an integral domain. The strategy of the proof is the following. First, we introduce a geometric model studied by H. Weyl \cite{W39} and C. D. Concini and C. Procesi \cite{CP76} to characterize the ring $\mathcal{Z}_n(\mathcal{P})$ as a $\operatorname{GL}(n,\mathbb{K})$-module. Then, we transfer the integrality of the ring $\mathcal{Z}_n(\mathcal{P})$ to another ring $K_{n,p}^{\operatorname{GL}(n,\mathbb{K})}$ by an injective ring homomorphism. The homomorphism is recovered by a long exact sequence of Lie group cohomology with values in $\operatorname{GL}(n,\mathbb{K})$-modules by Proposition \ref{propcohseq}. In the end, we prove that the ring $K_{n,p}$ is integral, which will complete the proof of the theorem.
\subsection{A geometric model for $\mathcal{Z}_n(\mathcal{P})$}
Let us introduce a geometric model to characterize $\mathcal{Z}_n(\mathcal{P})$.
 Let $M_{n,p} = (\mathbb{K}^n \times \mathbb{K}^{n*})^p$ be the space of $p$ vectors in $\mathbb{K}^n$ and $p$ co-vectors in $\mathbb{K}^{n*}$.
\begin{notation}
\label{groupaction}
Let $a_i = \left(a_{i,1},...,a_{i,n}\right)^T$, $b_i = \sum_{l=1}^n b_{i,l} \sigma_l$ where $a_{i,l}, b_{i,l}\in \mathbb{K}$, $\sigma_l \in \mathbb{K}^{n*}$ and $\sigma_l(a_i) = a_{i,l}$.
We define the product between a vector $a_i$ in $\mathbb{K}^n$ and a co-vector $b_j$ in $\mathbb{K}^{n*}$ by
   \begin{equation}
   \left<a_i|b_j\right> := b_j(a_i) = \sum_{k=1}^n a_{i,l} \cdot b_{j,l}.
   \end{equation}

 The group $\operatorname{GL}(n,\mathbb{K})$ acts naturally on the vectors and the covectors by
 $$g \circ a_i := g \cdot \left(a_{i,1},...,a_{i,n}\right)^T,$$
 $$g \circ b_j :=  \left(b_{i,1},...,b_{i,n}\right) \cdot (g^{-1}) \cdot \left(\sigma_1,...,\sigma_n\right)^T$$
  where $T$ is the transpose of the matrix.
   When we consider the action on their products, we write $b_j = \left(b_{i,1},...,b_{i,n}\right)^T$ in column as $a_i$, then
   $$g \circ b_j :=  (g^{-1})^T \cdot \left(b_{i,1},...,b_{i,n}\right)^T.$$
   For any $g \in \operatorname{GL}(n,\mathbb{K})$, $a,b \in \mathbb{K}[M_{n,p}]$,
   $$g \circ (a\cdot b) := (g \circ a) \cdot (g \circ b).$$
   It induces a $\operatorname{GL}(n,\mathbb{K})$ action on $\mathbb{K}[M_{n,p}]$ satisfying:
 \begin{itemize}
   \item For any $g \in \operatorname{GL}(n,\mathbb{K})$, $a,b \in \mathbb{K}[M_{n,p}]$, we have
   $$g \circ (a+b) = g \circ a + g \circ b,$$
   \item For any $g_1,g_2 \in \operatorname{GL}(n,\mathbb{K})$, $a \in \mathbb{K}[M_{n,p}]$, we have
   $$g_1 \circ (g_2 \circ a) = (g_1 \cdot g_2) \circ a.$$
 \end{itemize}

Then the polynomial ring $\mathbb{K}[M_{n,p}]$ is a $\operatorname{GL}(n,\mathbb{K})$-module.
\end{notation}

Let $B_{n\mathbb{K}}$ be the subring of $\mathbb{K}[M_{n,p}]$ generated by $\{\left<a_i|b_j \right>\}_{i=1,j=1}^p$.
We denote the $\operatorname{GL}(n,\mathbb{K})$ invariant ring of $\mathbb{K}[M_{n,p}]$ by $\mathbb{K}[M_{n,p}]^{\operatorname{GL}(n,\mathbb{K})}$.
Since $\left<a_i|b_j\right> \in \mathbb{K}[M_{n,p}]$ is invariant under $\operatorname{GL}(n,\mathbb{K})$ action, we have $B_{n\mathbb{K}} \subseteq \mathbb{K}[M_{n,p}]^{\operatorname{GL}(n,\mathbb{K})}$.
Moreover, C. D. Concini and C. Procesi proved that
   \begin{thm}\label{thmpro}
{\sc[C. D. Concini and C. Procesi \cite{CP76}  \footnotemark[1]]} \footnotetext[1]{Thanks for the reference provided by J. B. Bost.}
$B_{n\mathbb{K}} = \mathbb{K}[M_{n,p}]^{\operatorname{GL}(n,\mathbb{K})}$.
\end{thm}
Since $\mathbb{K}[M_{n,p}]$ is an integral domain, they obtained the following corollary.
\begin{cor}
{\sc[C. D. Concini and C. Procesi \cite{CP76}]} The subring $B_{n\mathbb{K}}$ is an integral domain.
\end{cor}
H. Weyl describe $B_{n\mathbb{K}}$ as a quotient ring.
\begin{thm}
{\sc[H. Weyl \cite{W39}]} All the relations in $B_{n\mathbb{K}}$ are generated by $R = \{ f \in B_{n\mathbb{K}} \; |  \; f = \det \left(\begin{array}{lcr}
                                        \left<a_{i_1}| b_{j_1}\right> & ... & \left<a_{i_1} |b_{j_{n+1}}\right> \\
                                        ... & ... & ... \\
                                        \left<a_{i_{n+1}} |b_{j_1}\right> & ... & \left<a_{i_{n+1}}| a_{j_{n+1}}\right>
                                      \end{array} \right), \forall i_{k}, j_{l} = 1,...,p \}$.
\end{thm}
\begin{remark}
In other words, let $W$ be the polynomial ring $\mathbb{K}[\{x_{i,j}\}_{i,j=1}^p]$,
$r = \{ f \in W \; |  \; f = \det \left(\begin{array}{lcr}
                                        x_{i_1 ,j_1} & ... & x_{i_1 ,j_{n+1}} \\
                                        ... & ... & ... \\
                                       x_{i_{n+1}, j_1} & ... & x_{i_{n+1}, j_{n+1}}
                                      \end{array} \right), \forall i_{k}, j_{l} = 1,...,p \}$, let $T$ be the ideal of $W$ generated by $r$, then we have $B_{n\mathbb{K}} \cong W/T$.
\end{remark}

Let us recall that $\mathcal{Z}_n(\mathcal{P}) = \mathcal{Z}(\mathcal{P})/R_n(\mathcal{P})$ is the rank $n$ swapping ring where $\mathcal{P} = \{x_1,...,x_p\}$.
When we identify $a_i$ with $x_i$ on the left and $b_i$ with $x_i$ on the right of the pairs of points in $\mathcal{Z}_n(\mathcal{P})$, we obtain the main result of this subsection below.
\begin{thm}
\label{inv1}
 Let $\mathcal{Z}_n(\mathcal{P})$ be the rank $n$ swapping ring. Let $S_{n\mathbb{K}}$ be the ideal of $B_{n\mathbb{K}}$ generated by $\{\left<a_i|b_i\right>\}_{i=1}^p$, then $B_{n\mathbb{K}}/{S_{n\mathbb{K}}}  \cong \mathcal{Z}_n(\mathcal{P})$.
\end{thm}
\label{thm4.4}
\subsection{Proof of the secound main result}
\begin{thm}{\sc[\textbf{Second main result}]}
\label{intdom}
For $n>1$,
$\mathcal{Z}_n(\mathcal{P})$ is an integral domain.
\end{thm}
Firstly, let us first consider the following $\operatorname{GL}(n,\mathbb{K})$-modules:
   \begin{enumerate}
   \item Let $L$ be the ideal of $\mathbb{K}[M_{n,p}]$ generated by $(\{\left<a_i| b_i \right>\}_{i=1}^p)$,
   \item  let $K_{n,p}$ be the quotient ring $\mathbb{K}[M_{n,p}]/L$,
   \item let $S_{n\mathbb{K}}$ be the ideal of $B_{n\mathbb{K}}$ generated by $\{\left<a_i|b_i \right>\}_{i=1}^p$.
   \end{enumerate}

Thus there is an exact sequence of $\operatorname{GL}(n,\mathbb{K})$-modules (the right arrows are not only module homomorphisms, but also ring homomorphisms):
\begin{equation}
\label{exactseq}
0 \rightarrow L \rightarrow \mathbb{K}[M_{n,p}] \rightarrow K_{n,p} \rightarrow 0.
\end{equation}
By Lie group cohomology \cite{CE48}, the exact sequence above induces the long exact sequence :
\begin{equation}
\label{long}
0 \rightarrow L^{\operatorname{GL}(n,\mathbb{K})} \rightarrow \mathbb{K}[M_{n,p}]^{\operatorname{GL}(n,\mathbb{K})} \rightarrow K_{n,p}^{\operatorname{GL}(n,\mathbb{K})} \rightarrow H^1(\operatorname{GL}(n,\mathbb{K}),L)\rightarrow... .
\end{equation}

\begin{lem}
\label{lemgit}
Let $S$ be the finite subset $\{\left<a_i|b_i \right>\}_{i=1}^p$. Let $\mathbb{K}$ be a field of characteristic $0$. Then $$(\mathbb{K}[M_{n,p}]\cdot S)^{\operatorname{GL}(n,\mathbb{K})} = \mathbb{K}[M_{n,p}]^{\operatorname{GL}(n,\mathbb{K})}\cdot S.$$
\end{lem}
\begin{proof}
The proof follows from Weyl's unitary trick.
Let $$\operatorname{U}(n) = \{g \in \operatorname{GL}(n,\mathbb{K})\;|\; g\cdot \bar{g}^T = I\}.$$ We want to prove that $$\left(\mathbb{K}[M_p]\cdot S\right)^{\operatorname{U}(n)} = \mathbb{K}[M_p]^{\operatorname{U}(n)}\cdot S.$$
Notice first that one inclusion is obvious $$\left(\mathbb{K}[M_p]\cdot S\right)^{\operatorname{U}(n)} \supseteq \mathbb{K}[M_p]^{\operatorname{U}(n)}\cdot S.$$
      We next prove the other inclusion: $$\left(\mathbb{K}[M_p]\cdot S\right)^{\operatorname{U}(n)} \subseteq \mathbb{K}[M_p]^{\operatorname{U}(n)}\cdot S.$$ For this, let $\D g$ be a Haar measure on $\operatorname{U}(n)$. Let $x$ belongs to $\left(\mathbb{K}[M_p]\cdot S\right)^{\operatorname{U}(n)}$.
      We represent $x$ by $\sum_{l=1}^k t_l \cdot  s_l$, where $t_l \in \mathbb{K}[M_p]$ and $s_l \in  S$.
       Since $S \subseteq \mathbb{K}[M_p]^{\operatorname{GL}(n,\mathbb{K})} \subseteq \mathbb{K}[M_p]^{\operatorname{U}(n)}$, for any $g\in \operatorname{U}(n)$,
      $g \circ s_l = s_l$.
       Thus we have $$x = g \circ x = \sum_{l=1}^k (g \circ t_l) \cdot \left(g \circ s_l \right)  = \sum_{l=1}^k (g \circ t_l) \cdot s_l.$$
By integrating over $\operatorname{U}(n)$: \begin{equation}
\label{eqgr}
g \circ x  =  \int_{ \operatorname{U}(n)} \sum_{l=1}^k (g \circ t_l) \cdot s_l \;\D g = \sum_{l=1}^k \left(\int_{ \operatorname{U}(n)}  g \circ t_l \; \D g \right) \cdot s_l,
\end{equation}
where $$b_l = \int_{ \operatorname{U}(n)}  g \circ t_l  \; \D g \in \mathbb{K}[M_p].$$ For any $g_1$ in $\operatorname{U}(n)$, we have
\begin{equation}
\begin{aligned}
&g_1 \circ b_l = \int_{ \operatorname{U}(n)}  g_1 \circ \left(g \circ t_l\right) \;\D g = \int_{ \operatorname{U}(n)}  \left(\left(g_1 \cdot g\right) \circ t_l\right)  \;\D g \\&= \int_{ \operatorname{U}(n)}  \left(\left(g_1 \cdot g\right) \circ t_l\right) \; \D \left(g_1\cdot g \right) = b_l.
\end{aligned}
\end{equation}
Thus $b_l$ belongs to $\mathbb{K}[M_p]^{\operatorname{U}(n)}$,
 $x$ belongs to  $\mathbb{K}[M_p]^{\operatorname{U}(n)}\cdot S$. Hence
$$\left(\mathbb{K}[M_p]\cdot S\right)^{\operatorname{U}(n)} \subseteq \mathbb{K}[M_p]^{\operatorname{U}(n)}\cdot S.$$
Therefore, we obtain $$\left(\mathbb{K}[M_p]\cdot S\right)^{\operatorname{U}(n)} = \mathbb{K}[M_p]^{\operatorname{U}(n)}\cdot S.$$
By extending the ground field $\mathbb{K}$ of $\operatorname{U}(n)$, the property is extended to $\operatorname{GL}(n,\mathbb{K})$. Therefore, we conclude that
$$(\mathbb{K}[M_{n,p}]\cdot S)^{\operatorname{GL}(n,\mathbb{K})} = \mathbb{K}[M_{n,p}]^{\operatorname{GL}(n,\mathbb{K})}\cdot S.$$
This complete the proof of the lemma.
\end{proof}

Then, the integrality of $\mathcal{Z}_n(\mathcal{P})$ is transferred to another ring $K_{n,p}^{\operatorname{GL}(n,\mathbb{K})}$ by the following proposition.
\begin{prop}
\label{propcohseq}
There is a ring homomorphism $\theta: B_{n\mathbb{K}}/S_{n\mathbb{K}}\rightarrow K_{n,p}^{\operatorname{GL}(n,\mathbb{K})}$ induced from the long exact sequence:
\begin{equation}
0 \rightarrow L^{\operatorname{GL}(n,\mathbb{K})} \rightarrow \mathbb{K}[M_{n,p}]^{\operatorname{GL}(n,\mathbb{K})} \rightarrow K_{n,p}^{\operatorname{GL}(n,\mathbb{K})} \rightarrow H^1(\operatorname{GL}(n,\mathbb{K}),L)\rightarrow.... ,
\end{equation}
 which is injective.
\end{prop}
\begin{proof}
By Theorem \ref{thmpro}, we have $\mathbb{K}[M_{n,p}]^{\operatorname{GL}(n,\mathbb{K})} = B_{n\mathbb{K}}$. By Lemma \ref{lemgit}, we have
$$L^{\operatorname{GL}(n,\mathbb{K})} = \left(\mathbb{K}[M_{n,p}]\cdot (\{\left<a_i|b_i \right>\}_{i=1}^p) \right)^{\operatorname{GL}(n,\mathbb{K})} = \mathbb{K}[M_{n,p}]^{\operatorname{GL}(n,\mathbb{K})} \cdot \left(\{\left<a_i|b_i \right>\}_{i=1}^p \right) = B_{n\mathbb{K}} \cdot (\{\left<a_i|b_i \right>\}_{i=1}^p) = S_{n\mathbb{K}}.$$
Hence Long exact sequence \ref{long} becomes into:
\begin{equation}
0 \rightarrow S_{n\mathbb{K}} \rightarrow B_{n\mathbb{K}} \rightarrow K_{n,p}^{\operatorname{GL}(n,\mathbb{K})} \rightarrow H^1(\operatorname{GL}(n,\mathbb{K}),L)\rightarrow...
\end{equation}
Therefore, there is an injective module homomorphism $\theta$ from $B_{n\mathbb{K}}/S_{n\mathbb{K}}$ to $K_{n,p}^{\operatorname{GL}(n,\mathbb{K})}$. By definition of $\theta$ with respect to Exact sequence \ref{exactseq}, for any $a,b \in B_{n\mathbb{K}}/S_{n\mathbb{K}}$, we have $$\theta(a\cdot b)=\theta(a)\cdot \theta(b).$$
Thus the module homomorphism $\theta$ is also a ring homomorphism. As such, we conclude that the ring homomorphism $\theta$ is injective.
\end{proof}
Therefore, we only need to prove that $K_{n,p}$ is integral. This is the content of the following proposition.
\begin{prop}
\label{propKp}
For $n>1$, $K_{n,p}$ is an integral domain.
\end{prop}
\begin{proof}
We prove the proposition by induction on the number of the vectors or covectors $p$. Let us start with $p=1$. When $p = 1$,  $K_{n,1} = \mathbb{K}[M_{n,1}]/(\{\sum_{k=1}^n a_{1,k} \cdot b_{1,k}\})$.
Let us define the degree of a monomial in $\mathbb{K}[M_{n,1}]$ to be the sum of the degrees in all the variables. Let the degree of a polynomial $f$ in $\mathbb{K}[M_{n,1}]$ be the maximal degree of the monomials in $f$, denoted by $\deg(f)$. Suppose that $\sum_{k=1}^n a_{1,k} .b_{1,k}$ is a reducible polynomial in $\mathbb{K}[M_{n,1}]$, we have $$\sum_{k=1}^n a_{1,k} .b_{1,k} = g\cdot h,$$ where $g,h \in \mathbb{K}[M_{n,1}]$, $\deg(g) >0$ and $\deg(h) >0$. Since $\mathbb{K}[M_{n,1}]$ is an integral domain, $2 = \deg(g h) = \deg(g) + \deg(h)$, so we have $\deg(g) =\deg(h) = 1$.  Suppose that $$g = \lambda_0 + \lambda_1 \cdot  c_1+...+ \lambda_r \cdot c_r ,$$ $$h = \mu_0 + \mu_1 \cdot  d_1+...+ \mu_s \cdot d_s,$$ where $\lambda_1,..., \lambda_r, \mu_1,..., \mu_s$ are non zero elements in $\mathbb{K}$, $c_1,...,c_r$ ($d_1,...,d_s$ resp.) are different elements in $\{a_{1,k}, b_{1,k}\}_{k=1}^n$. Since there is no square in $g\cdot h$, we have $$\{c_1,...,c_r\} \cap \{d_1,...,d_s\} = \emptyset .$$
Because there are n monomials in $gh$, we obtain
$$r \cdot s = n.$$
Moreover, there are $2n$ variables in $g\cdot h$, we have
$$r + s = 2n,$$
thus $$r\cdot s \geq 2n-1.$$ Since $n>1$, we obtain that $$r\cdot s \geq 2n-1 >n = r\cdot s,$$ which is a contradiction. We therefore conclude that  $\sum_{k=1}^n a_{1,k} .b_{1,k}$ is an irreducible polynomial in $\mathbb{K}[M_{n,p}]$.
Since $\mathbb{K}[M_{n,1}]$ is an integral domain, we obtain that $K_{n,1}$ is an integral domain. Suppose that the proposition is true for $p = m \geq 1$. When $p = m+1$,
\begin{equation}
\begin{aligned}
&K_{n,m+1} = \mathbb{K}\left[\{a_{i,k}, b_{i,k}\}_{i,k=1}^{m+1,n}\right]/\left(\{\sum_{k=1}^n a_{i,k} .b_{i,k}\}_{i=1}^{m+1}\right) \\&= K_{n,m}\left[\{a_{{m+1},k}, b_{{m+1},k}\}_{k=1}^n\right/\left(\sum_{k=1}^n a_{{m+1},k} \cdot b_{{m+1},k}\right),
\end{aligned}
\end{equation}
we have $K_{n,m}$ is an integral domain by induction, thus $K_{n,m}[\{a_{{m+1},k}, b_{{m+1},k}\}_{k=1}^n]$ is an integral domain. By the above argument, the polynomial $\sum_{k=1}^n a_{{m+1},k} \cdot b_{{m+1},k}$ is an irreducible polynomial over $\mathbb{K}[\{a_{{m+1},k}, b_{{m+1},k}\}_{k=1}^n]$. Moreover, $a_{{m+1},k}, b_{{m+1},k}$($k=1,...,n$) are not variables that appear in $K_{n,m}$, so $\sum_{k=1}^n a_{{m+1},k} \cdot b_{{m+1},k}$ is an irreducible polynomial over $K_{n,m}[\{a_{{m+1},k}, b_{{m+1},k}\}_{k=1}^n] $. Hence $K_{n,m+1}$ is an integral domain.

We therefore conclude that $K_{n,p}$ is an integral domain for any $p \geq 1$ and $n>1$.
\end{proof}
\begin{proof}[Proof of Theorem 4.7]
  By Proposition \ref{propKp}, the ring $K_{n,p}$ is an integral domain, we deduce that the invariant ring $K_{n,p}^{\operatorname{GL}(n,\mathbb{K})}$ is an integral domain. By Proposition \ref{propcohseq}, there is an injective ring homomorphism $\theta$ from $B_{n\mathbb{K}}/S_{n\mathbb{K}} $ to $K_{n,p}^{\operatorname{GL}(n,\mathbb{K})}$, so $B_{n\mathbb{K}}/S_{n\mathbb{K}}$ is an integral domain. Moreover, by Theorem \ref{inv1} $\mathcal{Z}_n(\mathcal{P}) \cong B_{n\mathbb{K}}/S_{n\mathbb{K}}$, we conclude that for $n>1$, the rank $n$ swapping ring $\mathcal{Z}_n(\mathcal{P})$ is an integral domain.
\end{proof}
\begin{remark}
The ring $\mathcal{Z}_1(\mathcal{P})$ is not an integral domain, since $$D = xy\cdot yz = \det \left(\begin{array}{cc}
                                                                                   xy & xz \\
                                                                                   yy & yz
                                                                                 \end{array}\right)
$$ is zero in $\mathcal{Z}_1(\mathcal{P})$, but we have $xy$ and $yz$ are not zero in $\mathcal{Z}_1(\mathcal{P})$ whenever $x\neq y$, $y\neq z$ .
\end{remark}
\subsection{Rank $n$ swapping multifraction algebra of $\mathcal{P}$}
After Theorem \ref{intdom}, we define rank $n$ swapping multifraction algebra of $\mathcal{P}$ without any obstruction.
\begin{defn}{\sc[rank $n$ swapping fraction algebra of $\mathcal{P}$]}
Let $\mathcal{Q}_n(\mathcal{P})$ be the total fraction ring of $\mathcal{Z}_n(\mathcal{P})$. The {\em rank $n$ swapping fraction algebra} of $\mathcal{P}$ is the total fraction ring $\mathcal{Q}_n(\mathcal{P})$ equipped with the swapping bracket, denoted by $\left(\mathcal{Q}_n(\mathcal{P}),\{\cdot,\cdot\}\right)$.
\end{defn}

Let $\mathcal{CR}_n(\mathcal{P}) = \{\left. [x,y,z,t]=\frac{xz}{xt} \cdot \frac{yt}{yz} \in \mathcal{Q}_n(\mathcal{P}) \; \right \vert\; \forall x,y,z,t \in \mathcal{P},\; x\neq t,\; y \neq z\}$ be the set of all the cross fractions in $\mathcal{Q}_n(\mathcal{P})$.
Let $\mathcal{B}_n(\mathcal{P})$ be the sub fraction ring of $\mathcal{Q}_n(\mathcal{P})$ generated by $\mathcal{CR}_n(\mathcal{P})$.

Similar to Proposition \ref{thm2.11}, we have the following.
\begin{prop}
The sub fraction ring $\mathcal{B}_n(\mathcal{P})$ is closed under swapping bracket.
\end{prop}
\begin{defn}{\sc[rank $n$ swapping multifraction algebra of $\mathcal{P}$]}
Let $n\geq 2$, the {\em rank $n$ swapping multifraction algebra of $\mathcal{P}$} is the sub fraction ring $\mathcal{B}_n(\mathcal{P})$ equipped with the closed swapping bracket, denoted by $\left(\mathcal{B}_n(\mathcal{P}),\{\cdot,\cdot\}\right)$.
\end{defn}
Then the ring homomorphism $I$ from $\mathcal{B}(\mathcal{R})$ to $C^{\infty}(H_n(S))$ for any $n>1$, induces the homomorphism $I_n$
from $\mathcal{B}_n(\mathcal{R})$ to $C^{\infty}(H_n(S))$ by extending the following formula on generators to $\mathcal{B}_n(\mathcal{R})$:
\begin{equation}
I_n([x,y,z,t]) = \mathbb{B}_{\rho}(x,y,z,t),
\end{equation}
for any $[x,y,z,t] \in \mathcal{CR}_n(\mathcal{R})$. By the rank $n$ cross ratio condition, the homomorphism $I_n$ is well-defined. Then we rephrase Theorem \ref{lswapabg} as follows.
\begin{thm}
{\sc[F. Labourie \cite{L12}]} With the same conditions as in Theorem \ref{lswapabg}, for any $b_0,b_1 \in \mathcal{B}_n(\mathcal{R})$, we have $$\lim_{n\rightarrow \infty}\{I_n(b_0),I_n(b_1)\}_{S_n} = I_n\circ\{b_0,b_1\}.$$
\end{thm}
Hence the rank $n$ swapping multifraction algebra $\left(\mathcal{B}_n(\mathcal{R}),\{\cdot,\cdot\}\right)$ characterizes the Hitchin component $H_n(S)$ for a fixed $n>1$.

\section{Cluster $\mathcal{X}_{\operatorname{PGL}(2,\mathbb{R}),D_k}$-space}
Even though the rank $n$ swapping multifraction algebra $(\mathcal{B}_n(\mathcal{P}), \{\cdot, \cdot\})$ characterizes the $\operatorname{PSL}(n, \mathbb{R})$ Hitchin component asymptotically, we still have the rank $n$ swapping multifraction algebra $(\mathcal{B}_n(\mathcal{P}), \{\cdot, \cdot\})$ characterizes the related object---cluster $\mathcal{X}_{\operatorname{PGL}(n,\mathbb{R}),D_k}$-space without this asymptotic behavior where $D_k$ is a disc with $k$ special points on the boundary. We will show a simple case when $n=2$.
We show that the cluster dynamic of $\mathcal{X}_{\operatorname{PGL}(2,\mathbb{R}),D_k}$ can be demonstrated in the rank $2$ swapping algebra. As a byproduct, we reprove that the Fock-Goncharov Poisson bracket for $\mathcal{X}_{\operatorname{PGL}(2,\mathbb{R}),D_k}$ is independent of the ideal triangulation.
\subsection{Cluster $\mathcal{X}_{\operatorname{PGL}(2,\mathbb{R}),D_k}$-space and rank $2$ swapping algebra}
Let $S$ be an oriented surface with non-empty boundary and a finite set $P$ of special points on boundary, considered modulo isotopy. In \cite{FG06}, Fock and Goncharov introduced the moduli space $\mathcal{X}_{G,S}$($\mathcal{A}_{G,S}$ resp.) which is a pair $(\nabla, f)$, where $\nabla$ is a flat connection on the principal $G$ bundle on the surface $S$, $f$ is a flat section of $\partial S\backslash P$ with values in the flag variety $G/B$ (decorated flag variety $G/U$ resp.). They found that the pair of two moduli spaces $(\mathcal{X}_{G,S}, \mathcal{A}_{G^L,S})$ is equipped with a cluster ensemble structure. Particularly, the moduli space $\mathcal{X}_{G,S}$ is called the {\em cluster $\mathcal{X}_{G,S}$-space}. Moreover, each one of the moduli spaces $\mathcal{X}_{G,S}, \mathcal{A}_{G,S}$ is equipped with a positive structure. When the set $P$ is empty, the hole on the surface $S$ should be regarded as the puncture, the positive part of $\mathcal{X}_{\operatorname{PGL}(2,\mathbb{R}),S}$ is related to the Teichm\"uller space of $S$, and the positive part of $\mathcal{A}_{\operatorname{SL}(2,\mathbb{R}),S}$ is related to Penner's decorated Teichm\"uller space \cite{P87}. The fact that Penner's decorated Teichm\"uller space is related to a cluster algebra was independently observed by M. Gekhtman, M. Shapiro, and A. Vainshtein \cite{GSV05}.

When $D_k$ is a disc with $k$ special points on the boundary, the generic cluster $\mathcal{X}_{\operatorname{PGL}(2,\mathbb{R}),D_k}$-space corresponds to the generic configuration space $\operatorname{Conf}_{2,k}$ of $k$ flags in $\mathbb{RP}^1$ up to projective transformations. Given a generic configuration of $k$ flags $m^1,y^1,z^1,t^1,n^1,x^1,...$ in $\mathbb{RP}^1$, let $P_k$ be the associated convex $k$-gon with $k$ vertices $m,y,z,t,n,x,...$ as illustrated in Figure \ref{Pk}. The ideal triangulation of $D_k$ corresponds to the triangulation of $P_k$. Given a triangulation $\mathcal{T}$ of the $k$-gon $P_k$, for any pair of triangles $(\Delta_{xyz}, \Delta_{xzt})$ of $\mathcal{T}$ where $x,y,z,t$ are anticlockwise ordered, the Fock-Goncharov coordinate \cite{FG06} corresponding to the inner edge $xz$ is
$$X_{xz} = -\frac{\Omega\left(\hat{y}^1 \wedge \hat{z}^1 \right)}{\Omega\left(\hat{t}^1\wedge \hat{z}^1\right)} \cdot \frac{\Omega\left(\hat{t}^1 \wedge\hat{x}^1\right)}{\Omega\left(\hat{y}^1\wedge \hat{x}^1\right)}.$$
By definition $X_{xz}=X_{zx}$, so there are $k-3$ different coordinates.
\begin{figure}\centering
\includegraphics[width=4in]{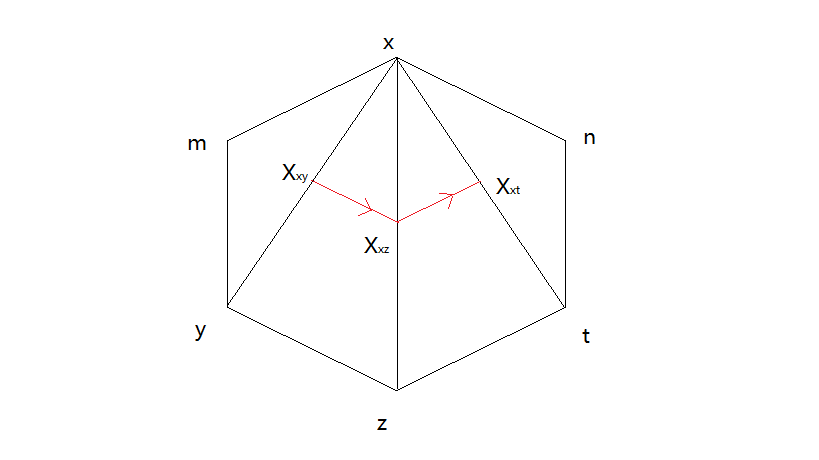}
\caption{\label{Pk} Fock-Goncharov coordinates for the triangulation $\mathcal{T}$.}
\end{figure}
\begin{defn}
\label{FGA}
{\sc[Fock-Goncharov algebra]}
Let $\mathcal{A}(\mathcal{T})$ be the fraction ring generated by $k-3$ Fock-Goncharov coordinates for the triangulation $\mathcal{T}$, the natural Fock-Goncharov Poisson bracket $\{\cdot,\cdot\}_2$ is defined on the fraction ring $\mathcal{A}(\mathcal{T})$ by extending the following map on the generators:
$$
\{X_{ab}, X_{cd}\}_2 = \varepsilon_{ab,cd} \cdot X_{ab} \cdot X_{cd}
$$
for any inner edges $ab$, $cd$, where the value of $\varepsilon_{ab,cd}$ only depend on the anticlockwise orientation of $P_k$ as illustrated in Figure \ref{Pk}. More precisely, $\varepsilon_{ab,cd} =1$($\varepsilon_{ab,cd} =-1$ resp.) when $a=c$ and $\Delta_{abd}$ is a triangle of $\mathcal{T}$ such that $a,b,d$ are ordered anticlockwise(clockwise resp.) in $P_k$; otherwise $\varepsilon_{ab,cd} =0$.

The {\em Fock-Goncharov algebra} of $\mathcal{T}$ is a pair $(\mathcal{A}(\mathcal{T}),\{\cdot,\cdot\}_2)$.
\end{defn}

\begin{defn}
Let $\mathcal{P}$ be the vertices of the convex $k$-gon $P_k$.
We define an injective ring homomorphism $\theta_{\mathcal{T}}$ from $\mathcal{A}(\mathcal{T})$ to $\mathcal{B}_2(\mathcal{P})$, by extending the following map on the generators:
\begin{equation}
\theta_{\mathcal{T}}(X_{xz}) := -\frac{yz}{tz} \cdot \frac{tx}{yx}
\end{equation}
for any inner edge of $\mathcal{T}$.
\end{defn}

\begin{thm}
\label{thmT}
The injective ring homomorphism $\theta_{\mathcal{T}}$ is Poisson with respect to the Poisson bracket $\{\cdot,\cdot\}_2$ and the swapping bracket $\{\cdot,\cdot\}$.
\end{thm}

\begin{proof}
By direct calculations, for any inner edge $ab$ of the triangulation $\mathcal{T}$, we have
\begin{equation}
\frac{\left\{ab, \frac{yz}{tz} \cdot \frac{tx}{yx}\right\}}{ab\cdot \frac{yz}{tz} \cdot \frac{tx}{yx}}
=\begin{cases}
    1   &  \text{if $ab = zx$;} \\
    -1   &  \text{if $ab = xz$;}\\
    0   & \text{otherwise.}
 \end{cases}
\end{equation}
The theorem follows from the above equation and the Leibniz's rule.
\end{proof}
\subsection{Cluster dynamic in rank $2$ swapping algebra}
The cluster $\mathcal{X}$-space is introduced by Fock and Goncharov \cite{FG06} by using the same set-up as the cluster algebra \cite{FZ02}. We consider the case for the cluster $\mathcal{X}_{\operatorname{PGL}(2,\mathbb{R}),D_k}$-space.
\begin{defn}
{\sc[cluster $\mathcal{X}_{\operatorname{PGL}(2,\mathbb{R}),D_k}$-space \cite{FG04} \cite{FG06}]}
Let $I_{\mathcal{T}}$ be the set of $k-3$ inner edges of the triangulation $\mathcal{T}$ of $D_k$. The function $\varepsilon$ from $I_{\mathcal{T}}\times I_{\mathcal{T}}$ to $\mathbb{Z}$ is defined as in Definition \ref{FGA}, a {\em seed} is $\mathbf{I}_{\mathcal{T}} = (I_{\mathcal{T}}, \varepsilon)$.

A {\em mutation} at the edge $e \in I_{\mathcal{T}}$ changes the seed $\mathbf{I}_{\mathcal{T}}$ to a new one $\mathbf{I}_{\mathcal{T}'} = (I_{\mathcal{T}'}, \varepsilon')$, where the edge $e$ of the triangulation $\mathcal{T}$ is changed into the edge $e'$ of $\mathcal{T}$ by a flip illustrated in Figure \ref{muk}. We identify $I_{\mathcal{T}}$ with $I_{\mathcal{T}'}$ by identifying $e$ with $e'$, where
$$ \varepsilon_{i,j}'=\left\{
\begin{aligned}
-\varepsilon_{i,j}\; if \; e \in \{i,j\};  \\
\varepsilon_{i,j} + \varepsilon_{i,e} \max\{0, \varepsilon_{i,e} \varepsilon_{e,j}\}  \; if \; e \notin \{i,j\}.
\end{aligned}
\right.
$$
A {\em cluster transformation} is a composition of mutations and automorphisms of seeds.

We assign to the seed $\mathbf{I}_{\mathcal{T}}$ ($\mathbf{I}_{\mathcal{T}'}$ resp.) the split tori $\mathbb{T}_{\mathcal{T}}$($\mathbb{T}_{\mathcal{T}'}$ resp.) associated to the the Fock-Goncharov coordinates $\{X_i\}_{i\in I_{\mathcal{T}}}$ ($\{X_i'\}_{i\in I_{\mathcal{T}'}}$ resp.). Then
the transition function from $\mathbb{T}_{\mathcal{T}}$ to $\mathbb{T}_{\mathcal{T}'}$ is
$$\mu_e(X_i')= g_e(X_i)=\left\{
\begin{aligned}
&X_i(1 + X_e)^{-\varepsilon_{i,e}}\; &if \; e\neq i,  \varepsilon_{i,e}\leq 0;  \\
&X_i(1 + X_e^{-1})^{-\varepsilon_{i,e}}  \; &if \; e\neq i, \varepsilon_{i,e} >0; \\
&X_e^{-1} \; & if \; i=e.
\end{aligned}
\right.
$$
Any two triangulations are related by a composition of flips, therefore any two split tori are also related by a composition of the rational functions as above.

The {\em cluster $\mathcal{X}_{\operatorname{PGL}(2,\mathbb{R}),D_k}$-space} is obtained by gluing all the possible algebraic tori $\mathbb{T}_{\mathcal{T}}$ according to the transition functions described as above.
\end{defn}

\begin{figure}\centering
\includegraphics[width=4in]{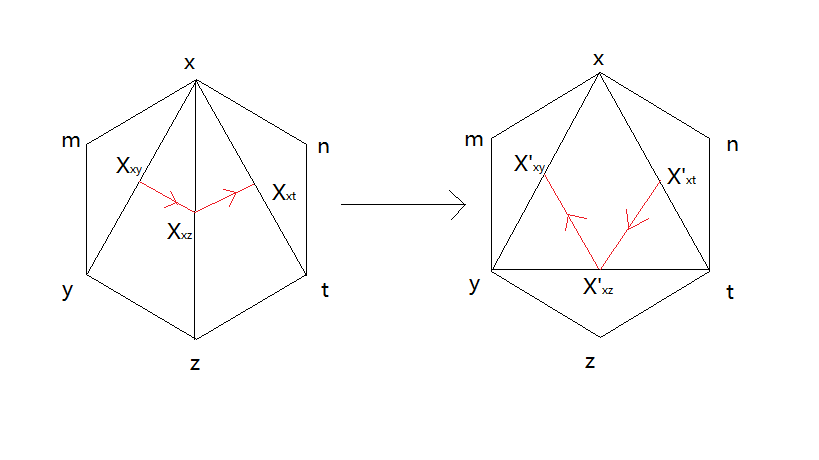}
\caption{\label{muk} Cluster transformation at $xz$.}
\end{figure}

We show the cluster dynamic of $\mathcal{X}_{\operatorname{PGL}(2,\mathbb{R}),D_k}$ in the rank $2$ swapping algebra as follows.
\begin{lem}
\label{lemT}
The triangulation $\mathcal{T}'$ is the flip of $\mathcal{T}$ at the edge $e$.
Then $$\theta_{\mathcal{T}}\circ g_e(X_i) = \theta_{\mathcal{T}'}(X_i').$$
\end{lem}
\begin{proof}
Let us consider the case where $e=xz$ and the triangulations $\mathcal{T}, \mathcal{T}'$ is illustrated in Figure \ref{muk}.
For $i=xz$, we have
$$\theta_{\mathcal{T}'}(X_{xz}') = -\frac{zt}{xt}\cdot \frac{xy}{zy},$$
$$\theta_{\mathcal{T}}\circ g_e(X_{xz})=\theta_{\mathcal{T}}\left(\frac{1}{X_{xz}}\right) = -\frac{tz}{yz}\cdot \frac{yx}{tx}  .$$
By the rank $2$ swapping algebra relations:
$$yz\cdot zt\cdot ty + tz\cdot zy \cdot yt=0,$$
$$tx\cdot xy\cdot yt + yx\cdot xt \cdot ty=0,$$
we obtain
$$\theta_{\mathcal{T}}\circ g_e(X_{xz}) = \theta_{\mathcal{T}'}(X_{xz}').$$

For $i=xy$, we have
$$\theta_{\mathcal{T}'}(X_{xy}') = -\frac{my}{ty}\cdot \frac{tx}{mx},$$
$$\theta_{\mathcal{T}}\circ g_e(X_{xy})=\theta_{\mathcal{T}}\left(X_{xy}\cdot\left(1+X_{xz}^{-1}\right)^{-1}\right) = -\frac{my}{zy}\cdot \frac{zx}{mx} \left(1 -\frac{tz}{yz}\cdot \frac{yx}{tx}\right)^{-1}  .$$
By the rank $2$ swapping algebra relation:
$$yz\cdot ty\cdot zx + yx\cdot tz \cdot zy - yz \cdot zy \cdot tx =0,$$
 we obtain
$$\theta_{\mathcal{T}}\circ g_e(X_{xy}) = \theta_{\mathcal{T}'}(X_{xy}').$$
We have same results for the other inner edges and the other cases different from the one illustrated in Figure \ref{muk}, by the similar arguments. We therefore conclude that $$\theta_{\mathcal{T}}\circ g_e(X_i) = \theta_{\mathcal{T}'}(X_i').$$
\end{proof}
\begin{prop}
The homomorphism $\mu_e$ preserves the Poisson bracket, so the Poisson bracket $\{\cdot,\cdot\}_2$ does not depend on the triangulation $\mathcal{T}$.
\end{prop}
\begin{proof}
We need to prove that $$\{\mu_e(X_i'),\mu_e(X_j')\}_2 = \mu_e\left(\{X_i',X_j'\}_2\right),$$
which is equivalent to $$\{g_e(X_i),g_e(X_j)\}_2 = \varepsilon_{i,j}' \cdot g_e(X_i)\cdot g_e(X_j).$$

Since $\theta_{\mathcal{T}'}$ is injective, we only need to prove that
$$\theta_{\mathcal{T}}\circ \{g_e(X_i),g_e(X_j)\}_2 = \theta_{\mathcal{T}}\left(\varepsilon_{i,j}' \cdot g_e(X_i)\cdot g_e(X_j)\right).$$

By Theorem \ref{thmT} and Lemma \ref{lemT}, we have
\begin{equation}
\begin{aligned}
&\theta_{\mathcal{T}}\circ \{g_e(X_i),g_e(X_j)\}_2 = \{\theta_{\mathcal{T}}\circ g_e(X_i), \theta_{\mathcal{T}}\circ g_e(X_j)\}
\\&=  \{\theta_{\mathcal{T}'}(X_i'), \theta_{\mathcal{T}'}(X_j')\} =  \theta_{\mathcal{T}'} \circ \{X_i', X_j'\}_2
\\&= \varepsilon_{i,j}' \cdot \theta_{\mathcal{T}'}(X_i) \cdot \theta_{\mathcal{T}'}(X_j) = \theta_{\mathcal{T}}\left(\varepsilon_{i,j}' \cdot g_e(X_i)\cdot g_e(X_j)\right).
\end{aligned}
\end{equation}
We therefore conclude that the homomorphism $\mu_e$ preserves the Poisson bracket.
\end{proof}

\begin{remark}
For $n$ in general, the generalized injective ring homomorphism $\theta_{T_n}$ is shown in \cite{Su15}, where the set $\mathcal{P}$ has $(n-1)\cdot k$ elements, each flag of $\mathbb{RP}^{n-1}$ corresponds to $n-1$ points near each other on the boundary $S^1$. We expect to glue the rank $n$ swapping algebras for the purpose of characterizing $\mathcal{X}_{\operatorname{PGL}(n,\mathbb{R}), S}$ for the surface case.
\end{remark}

\begin{acknowledgements}
This article is the first part of my thesis \cite{Su14} under the direction of Fran\c cois Labourie at university of Paris-Sud.
I am deeply indebted to Fran\c cois Labourie for suggesting the subject and for the guidance without which this work would not have been accomplished.
I also thank to Vladimir Fock for discussions and for suggesting many improvements to this article.
I also thank Jean Beno\^it Bost and Laurent Clozel for suggesting relevant references. Finally,
I am grateful to my home institute University of Paris-Sud, Max Planck Institute for
Mathematics and Yau Mathematical Sciences Center at Tsinghua University for their hospitality.
\end{acknowledgements}


\begin{thebibliography}{PTW02}
\bibitem[AB83]{AB83} M. Atiyah, R. Bott,
\emph{The Yang-Mills equations over Riemann surfaces}, Philos. Trans. Roy. Soc. London Ser. A, 308(1505):523-615, 1983.

\bibitem[CE48]{CE48} C. Chevalley and S. Eilenberg,
\emph{Cohomology Theory Of Lie Groups And Lie Algebras} , Transactions of the American Mathematical Society, Vol. 63, No. 1 (1948), 85-124.

\bibitem[CP76]{CP76} C. D. Concini, C. Procesi,
\emph{A Characteristic Free Approach to Invariant Theory}, Advances in Mathematics 21, 330-354(1976).

\bibitem[DS85]{DS85} V. G. Drinfeld, V. V. Sokolov,
\emph{Lie algebras and equations of Korteweg-de Vries type}, J. Sov. Math. 30 (1985), 1975-2035.

\bibitem[L06]{L06} F. Labourie,
\emph{Anosov Flows, Surface Groups and Curves in Projective Space}, Inventiones Mathematicae 165 no. 1, 51--114 (2006).

\bibitem[L07]{L07} F. Labourie,
\emph{Cross Ratios, Surface Groups, SL(n,R) and Diffeomorphisms of the Circle}, Publications de l'IHES, n. 106, 139-213 (2007).

\bibitem[L12]{L12} F. Labourie,
\emph{Goldman algebra, opers and the swapping algebra}, arXiv:1212.5015.

\bibitem[FD14]{FD14} F. Bonahon, G. Dreyer,
\emph{Parametrizing Hitchin components}, Duke Math. J. 163 (2014), 2935-2975.

\bibitem[FG04]{FG04} V. V. Fock, A. B. Goncharov,
\emph{Moduli spaces of convex projective structures on surfaces}, Adv. Math (2004)
Volume: 208, Issue: 1, Pages: 249-273.

\bibitem[FG06]{FG06} V. V. Fock, A. B. Goncharov,
\emph{Moduli spaces of local systems and higher Teichm\:uller theory}, Inst. Hautes \'Etudes Sci. Publ. Math. (2006), no. 103, 1-211.

\bibitem[FG09]{FG09} V. V. Fock, A. B. Goncharov,
\emph{Cluster ensembles, quantization and the dilogarithm}, Annales scientifiques de l'ENS 42, fascicule 6 (2009), 865-930.

\bibitem[FV93]{FV93}  L. Faddeev, A. Y. Volkov,
\emph{Abelian Current Algebra and the Virasoro Algebra on the Lattice},  Physics Letters B, Volume 315, Issue 3-4, p. 311-318.

\bibitem[FZ02]{FZ02}  S. Fomin, A. Zelevinsky,
\emph{Cluster algebras. I: Foundations}, J. Amer. Math. Soc. 15 (2002), no. 2, 497每529.

\bibitem[G84]{G84} W. M. Goldman,
\emph{The Symplectic Nature of Fundamental Groups of Surfaces} , Adv. Math.
54 (1984), no. 2, 200-225.

\bibitem[G88]{G88} W. M. Goldman,
\emph{Topological components of spaces of representations},  Invent. Math.93 (1988), pp. 557?607.

\bibitem[GSV05]{GSV05} M. Gekhtman, M. Shapiro, A. Vainshtein,
\emph{Cluster algebras and Weil-Petersson forms}, Duke Math. J.Volume 127, Number 2 (2005), 291-311.
(1987), no. 2, 299每339.

\bibitem[Gu08]{Gu08} O. Guichard,
\emph{Composantes de Hitchin et repr\'esentations hyperconvexes de groupes de surface}, J. Differential Geom. Volume 80, Number 3 (2008), 391-431.

\bibitem[H92]{H92} N. J. Hitchin,
\emph{Lie groups and Teichm\"uller space} , Topology 31 (1992), no. 3, 449-473.

\bibitem[MFK94]{MFK94} D. Mumford, J. Fogarty, F. Kirwan,
\emph{Geometric invariant theory} , Ergebnisse der Mathematik und ihrer Grenzgebiete 34 (3rd ed.), Berlin, New York: Springer-Verlag.

\bibitem[KS13]{KS13} B. Khesin, F. Soloviev,
\emph{Integrability of higher pentagram maps}, Mathematische Annalen November 2013, Volume 357, Issue 3, pp 1005-1047.

\bibitem[L94]{L94} G. Lusztig,
\emph{Total positivity in reductive groups, in: Lie Theory and Geometry}, Progr. Math. vol. 123, Birkh\"auser Boston (1994), pp. 531每568.

\bibitem[L98]{L98} G. Lusztig,
\emph{Total positivity in partial flag manifolds}, Representation Theory 2 (1998), pp. 70每78.

\bibitem[P87]{P87} R. C. Penner,
\emph{The decorated Teichm邦ller space of punctured surfaces}, Communications in Mathematical Physics (1987), Volume 113, Issue 2, pp 299-339.

\bibitem[Se91]{Se91} G. Segal,
\emph{The geometry of the kdv equation}, Internat. J. Modern Phys. A 6 (1991), no. 16, 2859-2869.

\bibitem[SOT10]{SOT10} R. Schwartz, V. Ovsienko and S. Tabachnikov,
\emph{The Pentagram map: A discrete integrable system} , Communications in Mathematical Physics (2010), Volume 299, Issue 2, pp 409-446.

\bibitem[Su14]{Su14} Z. Sun,
\emph{Rank n swapping algebra and its applcations}, thesis at university of Paris-Sud, 2014.

\bibitem[Su1412]{Su1412} Z. Sun,
\emph{Swapping algebra, Virasoro algebra and discrete integrable system}, arXiv:1412.4330.

\bibitem[Su15]{Su15} Z. Sun,
\emph{Fock-Goncharov coordinates and rank $n$ swapping algebra}, arXiv:1503.00918.

\bibitem[Su1511]{Su1511} Z. Sun,
\emph{Quantization of rank $n$ swapping algebra}, in preparation.

\bibitem[T86]{T86} W. P. Thurston,
\emph{Minimal Stretch Maps Between Hyperbolic Surfaces}, preprint (1986), arXiv:math/9801039.

\bibitem[W39]{W39} H. Weyl,
\emph{The Classical Groups: Their Invariants and Representations} , Princeton university press,  (1939), 320 pages.

\end{thebibliography}
\end{document}